\documentclass[11pt]{amsart} 
\allowdisplaybreaks[4]

\usepackage[T1]{fontenc}

\usepackage{indentfirst}
\usepackage{hyperref}

\usepackage{bm}

\usepackage{amssymb}
\usepackage{verbatim}
\usepackage{hyperref}
\usepackage{color}
\usepackage{tikz}

\def\bint{{\ifinner\rlap{\bf\kern.30em--}
		\int\else\rlap{\bf\kern.35em--}\int\fi}\ignorespaces}

\def\sbint{{\ifinner\rlap{\bf\kern.32em--}
		\hspace{0.078cm}\int\else\rlap{\bf\kern.45em--}\int\fi}\ignorespaces}

\def\bz{\beta}

\def\bxi{{\bar\xi}}

\def\BMO{\mathop\mathrm{\,BMO}}
\def\Lip{\mathop\mathrm{\,Lip}}

\def\dist{{\mathop\mathrm{\,dist\,}}}

\def\eqref#1{(\ref{#1})}

\def\bi{{\int_{B(x,t)}}}
\def\BMO{\mathop\mathrm{\,BMO}}
\def\BLO{\mathop\mathrm{\,BLO}}

\numberwithin{equation}{section}
\newcommand{\distance}{\mathop{\mathrm{distance}}\nolimits}

\begin{document}
\setcounter{tocdepth}{1}

% define theorem environments
\newtheorem{theorem}{Theorem}    %[section]
\newtheorem{proposition}[theorem]{Proposition}
\newtheorem{corollary}[theorem]{Corollary}
\newtheorem{lemma}[theorem]{Lemma}
\newtheorem{sublemma}[theorem]{Sublemma}
\newtheorem{conjecture}[theorem]{Conjecture}
\newtheorem{claim}[theorem]{Claim}
\newtheorem{fact}[theorem]{Fact}
\newtheorem{observation}[theorem]{Observation}

\newtheorem{definition}{Definition}
\newtheorem{notation}[definition]{Notation}
\newtheorem{remark}[definition]{Remark}
\newtheorem{question}[definition]{Question}
\newtheorem{questions}[definition]{Questions}
\newtheorem{hypothesis}[definition]{Hypothesis}

\newtheorem{example}[definition]{Example}
\newtheorem{problem}[definition]{Problem}
\newtheorem{exercise}[definition]{Exercise}

 \numberwithin{theorem}{section}
 \numberwithin{definition}{section}
 \numberwithin{equation}{section}

 \newtheorem{thm}{Theorem}
 \newtheorem{cor}[thm]{Corollary}
 \newtheorem{lem}{Lemma}[section]
 \newtheorem{prop}[thm]{Proposition}
 \theoremstyle{definition}
 \newtheorem{dfn}[thm]{Definition}
 \theoremstyle{remark}
 \newtheorem{rem}{Remark}
 \newtheorem{ex}{Example}
 \newtheorem*{thA}{Theorem A}
 \newtheorem*{thB}{Theorem B}

\def\repair{\medskip\hrule\hrule\medskip}

\def\bff{\mathbf f}
\def\bE{\mathbf E}
\def\bF{\mathbf F}
\def\bK{\mathbf K}
\def\bP{\mathbf P}
\def\bx{\mathbf x}
\def\bi{\mathbf i}
\def\bk{\mathbf k}
\def\bt{\mathbf t}
\def\bc{\mathbf c}
\def\ba{\mathbf a}
\def\bw{\mathbf w}
\def\bh{\mathbf h}
\def\bn{\mathbf n}
\def\bg{\mathbf g}
\def\bc{\mathbf c}
\def\bs{\mathbf s}
\def\bp{\mathbf p}
\def\by{\mathbf y}
\def\bv{\mathbf v}
\def\be{\mathbf e}
\def\bu{\mathbf u}
\def\bm{\mathbf m}
\def\bxi{{\mathbf \xi}}
\def\bR{\mathbf R}
\def\by{\mathbf y}
\def\bz{\mathbf z}
\def\bfb{\mathbf b}
\def\bPhi{{\mathbf\Phi}}

\newcommand{\norm}[1]{ \|  #1 \|}

\def\scriptl{{\mathcal L}}
\def\scriptc{{\mathcal C}}
\def\scriptd{{\mathcal D}}
\def\scrapd{{\mathcal D}}
\def\scripts{{\mathcal S}}
\def\scriptq{{\mathcal Q}}
\def\scriptt{{\mathcal T}}
\def\scriptf{{\mathcal F}}
\def\scriptm{{\mathcal M}}
\def\calM{{\mathcal M}}
\def\scripti{{\mathcal I}}
\def\scriptr{{\mathcal R}}
\def\scriptb{{\mathcal B}}
\def\scripte{{\mathcal E}}
\def\scripta{{\mathcal A}}
\def\scriptn{{\mathcal N}}
\def\scriptv{{\mathcal V}}
\def\scriptz{{\mathcal Z}}
\def\scriptj{{\mathcal J}}
\def\scriptk{{\mathcal K}}
\def\scriptg{{\mathcal G}}
\def\scripth{{\mathcal H}}

\def\HLM{{\mathbb M}}
\def\mbbm{{\mathfrak M}}

\def\bk{\mathbf k}
\def\kernel{\operatorname{kernel}}
\def\dist{\operatorname{distance}\,}
\def\eps{\varepsilon}

\def\reals{\mathbb R}
\def\naturals{\mathbb N}
\def\integers{\mathbb Z}
\def\rationals{\mathbb Q}
\def\one{\mathbf 1}
\def\complex{{\mathbb C}\/}

\def\lt{{L^2}}

\def\three{\mathbf 3}
\def\four{\mathbf 4}

\def\bart{\bar t}
\def\barz{\bar z}
\def\barx{{\bar x}}
\def\bary{\bar y}
\def\barz{{\bar z}}
\def\bars{\bar s}
\def\barc{\bar c}
\def\baru{\bar u}
\def\barr{\bar r}

\def\distance{\operatorname{distance}}
\def\md{{\mathcal D}}

\def\lsharp{\Lambda^\sharp}
\def\lnatural{\Lambda^\natural}
\def\sS{{\mathcal S}}
\def\barsS{\overline{\mathbb S}}

\title[Marcinkiewicz integrals associated with Schr\"{o}dinger operator] {Marcinkiewicz integrals associated with Schr\"{o}dinger operator on Campanato type spaces}

\author{Qingying Xue$^{*}$}
\author{Jiali Yu}
% \author{Draft of work in progress  --- MC and ZZ }
\address{Qingying Xue:
	School of Mathematical Sciences \\
	Beijing Normal University \\
	Laboratory of Mathematics and Complex Systems \\
	Ministry of Education \\
	Beijing 100875 \\
	People's Republic of China}
\email{qyxue@bnu.edu.cn}

\address{Jiali Yu\\
       	School of Mathematical Sciences \\
       Beijing Normal University \\
       Laboratory of Mathematics and Complex Systems \\
       Ministry of Education \\
       Beijing 100875 \\
       People's Republic of China}
     \email{jialiyu@mail.bnu.edu.cn}

\begin{abstract}The boundedness of the Marcinkiewicz integrals associated with Schr\"{o}dinger operator from the localized Morrey-Campanato space to the localized Morrey-Campanato-BLO space is established. Similar results for the Marcinkiewicz integrals associated with the Schr\"{o}dinger operator  with rough kernels were also investigated.

\end{abstract}


\thanks{The authors were partly supported by the National Key R\&D Program of China (No. 2020YFA0712900) and NNSF of China (No. 12271041).
\thanks{$^{*}$ Corresponding author, e-mail address: qyxue@bnu.edu.cn}}

\date{\today}

 \maketitle

\section{Introduction} Consider the Schr\"{o}dinger operator $\mathcal{L}=-\Delta+V$ on $\mathbb{R}^n,\,n \geq 3$, where $\Delta$ is the Laplace operator and the potential $V$ is a nonnegative locally $L^s$ integrable function on $\mathbb{R}^n$ and $V$ belongs to the reverse H\"{o}lder class $\mathcal{B}_s$. Recall that $V$ belongs to the reverse H\"{o}lder class $\mathcal{B}_s$ for some $s>\frac{n}{2}$ if there exists a constant $C=C(s,V)>0$ such that
\begin{align}\label{ineq:ni holder}
	\left(\frac{1}{\left|B\right|}\int_{B}V(x)^sdx\right)^{\frac{1}{s}}\leq C\left(\frac{1}{\left|B\right|}\int_{B}V(x)dx\right),
\end{align}
for every ball $B\subset\mathbb{R}^n$. It was known from \cite{G} that $V\in\mathcal{B}_s\ (s>1)$ implies $V\in\mathcal{B}_{s+\varepsilon}$ for some $\varepsilon>0$. Therefore, $\mathcal{B}_s\,(s>n/2)$ is equivalent with $\mathcal{B}_s\,(s\geq n/2)$ and in turn one can define the reverse H\"{o}lder index of $V$ as $\sup\{s:V\in\mathcal{B}_s\}$. Throughout this paper, we always assume $0\neq V\in\mathcal{B}_n$. Under this assumption, there exists $q_0>n$ such that $V\in\mathcal{B}_{q_0}$. In \cite{S3}, Shen introduced the associated auxiliary function $\rho(x,V)=\rho(x)$ as follow:
\begin{align}\label{eq:admissible}
	\rho(x)=\sup\limits_{r>0}\left\{r:\frac{1}{r^{n-2}}\int_{B(x,r)}V(y)dy\leq1\right\},\,\,x\in\mathbb{R}^n.
\end{align}
It is clearly that $0<\rho(x)<\infty$ for any $x\in\mathbb{R}^n$.

The study of Schr\"{o}dinger operator attracts much attention in both Harmonic analysis and PDE. It was Fefferman \cite{F}, Shen \cite{S3} and Zhong \cite{Z2} who obtained the basic estimates for $\mathcal{L}$ and the $L^p\,(1<p<+\infty)$ boundedness of $\mathcal{L}$ with certain potentials which includes Schr\"{o}dinger-Riesz transforms $\nabla\mathcal{L}^{-\frac{1}{2}}$ on $\mathbb{R}^n$. For $V\in\mathcal{B}_{n/2}$ with $n\geq 3$, Dziuban\'nski et al \cite{DZ1,DZ2,DZ3} studied the BMO-type space BMO${_\mathcal{L}}(\mathbb{R}^n)$ and Hardy-type space $H{_\mathcal{L}^p}(\mathbb{R}^n)$ for $ p\in(n/(n+1),1]$. They proved that the dual space of $H{_\mathcal{L}^1}(\mathbb{R}^n)$ is BMO${_\mathcal{L}}(\mathbb{R}^n)$, where the space BMO${_\mathcal{L}}(\mathbb{R}^n)$ can be regarded as a special case of BMO-type space introduced by Duong and Yan \cite{DY2,DY1}. In particular, the estimates of variants of several classical operators were given on these spaces, such as the Littlewood-Paley $g$-function associated with $\mathcal{L}$ defined by
\begin{align}\label{def:gwithl}
	g(f)(x):=\left(\int_{0}^{\infty}\left|Q_t(f)(x)\right|^2\frac{dt}{t}\right)^{\frac{1}{2}},
\end{align}
where $Q_t(f)(x)=t^2\left(\frac{dT_s}{ds}\big|_{s=t^2}f\right)(x)$ with the integral kernels $Q_t(x,y)=t^2\frac{\partial k_s(x,y)}{\partial s}\big|_{s=t^2}$ and $\{T_t\}_{t>0}=\{e^{-t\mathcal{L}}\}_{t>0}$ with kernels $\{k_t(x, y)\}_{t>0}$ is the semigroup generated by $\mathcal{L}$.

Later on, Dong et al.\ \cite{DHL} demonstrated that the dual space of $H{_\mathcal{L}^p}(\mathbb{R}^n)$ is the localized Morrey-Campanato space $\mathcal{E}{_\mathcal{L}^{1/p-1,q}},$ where $q\in[1,+\infty)$ and $p\in(n/(n+\delta),1]$ for some $\delta\in(0,\infty)$.

In order to state more related results, let us recall the definition of the Morrey-Campanato space $\mathcal{E}^{\alpha,p}(\mathbb{R}^n)$ and a subspace of it. 

	\begin{definition}[\cite{C,HMY1}] (Morrey-Campanato type spaces) Let $\alpha\in(-\infty,1],\,p\in(0,\infty)$.\\
		{\rm (i) (Morrey-Campanato space)} A locally integrable function $f$ is said to belong to the Morrey-Campanato space $\mathcal{E}^{\alpha,p}(\mathbb{R}^n)$ if
		\begin{align*}
			\Vert f\Vert_{\mathcal{E}^{\alpha,p}}:=\sup\limits_{B\subset\mathbb{R}^n}\frac{1}{\left|B\right|^{\alpha/n}}\left(\frac{1}{\left|B\right|}\int_{B}\left|f(x)-f_B\right|^pdx\right)^{\frac{1}{p}}<\infty,
		\end{align*}
		where the supremum is taken over all balls $B\subset\mathbb{R}^n$ and $f_B=\frac{1}{\left|B\right|}\int_{B}f(x)dx.$\\ {\rm (ii) (Morrey-Campanato-BLO space)} A locally integrable function $f$ is said to be in the Morrey-Campanato-BLO space $\widetilde{\mathcal{E}}{^{\alpha,p}}(\mathbb{R}^n)$ if 
		\begin{align*}
			\Vert f\Vert_{\widetilde{\mathcal{E}}{^{\alpha,p}}}:=\sup\limits_{B\subset\mathbb{R}^n}\frac{1}{\left|B\right|^{\alpha/n}}\left(\frac{1}{\left|B\right|}\int_{B}\big[f(x)-\mathop{{\rm essinf}}\limits_{B} f\big]^pdx\right)^{\frac{1}{p}}<\infty,
		\end{align*}
		where the supremum is taken over all balls $B\subset\mathbb{R}^n$. 
\end{definition}

When $p\in[1,\infty)$, it was known that for $\alpha>0,\ \mathcal{E}^{\alpha,p}(\mathbb{R}^n)=\Lip_\alpha(\mathbb{R}^n)$ with equivalent norms; for $\alpha=0$, $\mathcal{E}^{\alpha,p}(\mathbb{R}^n)$ is equivalent to $\BMO(\mathbb{R}^n)$; and for $\alpha<0$, $\mathcal{E}^{\alpha,p}(\mathbb{R}^n)$ concides with Morrey space. It was Coifman and Rochberg \cite{CR} who first introduced the BLO space which is a subspace of BMO. Subsequently, Bennett \cite{B} gave an equivalent characterization of BLO space. Moreover,  a subspace of Morrey-Campanato space, which is called Morrey-Campanato-BLO space $\widetilde{\mathcal{E}}{^{\alpha,p}}(\mathbb{R}^n)$ was considered by Hu et al. in \cite{HMY1}. It deserves to mention that when $\alpha=0$ and $p\in[1,\infty)$, $\widetilde{\mathcal{E}}{^{\alpha,p}}(\mathbb{R}^n)$ concides with BLO space.

It is quite natural to consider the above spaces associated with the Schr\"{o}dinger operator $\mathcal{L}$. The BLO-type space BLO${_\mathcal{L}}(\mathbb{R}^n)$ was characterized by Gao et al.\,\cite{GJT} in the study of the maximal Riesz transforms and Hardy-Littlewood maximal operator associated with $\mathcal{L}$. Almost at the same time, Yang et al. \cite{YYZ1} gave a slightly different definition of BLO${_\mathcal{L}}(\mathbb{R}^n)$ and demonstrated their equivalence in \cite{YYZ2}.

Yang et al. \cite{YYZ2,YYZ3} extented the above spaces on $\mathbb{R}^n$ to space $\mathcal{X}$ of homogeneous type in the sense of Coifman and Weiss \cite{CW1,CW2} with a reverse doubling condition. On $\mathcal{X}$, they introduced the locialized BMO space BMO${_\rho^q}(\mathcal{X})$, the locialized BLO space BLO${_\rho^q}(\mathcal{X})$, the localized Morrey-Campanato space $\mathcal{E}{_\mathcal{D}^{\alpha,p}}(\mathcal{X})$ and the localized Morrey-Campanato-BLO space $\widetilde{\mathcal{E}}{_\mathcal{D}^{\alpha,p}}(\mathcal{X})$. Moreover, the boundedness of various operators on these spaces were obtained. One of the typical example is the Littlewood-Paley $g$-function associated with $\mathcal{L}$. These results can be applied in a wide range of settings. It is worth pointing out that when $\alpha=0,\ q\in[1,+\infty),\ \mathcal{X}=\mathbb{R}^n,$ and$\ \mathcal{D}=D_\rho:=\left\{B=B(x,r):x\in \mathbb{R}^n,r\geq \rho (x)\right\}$, BMO${_\rho^{q}}(\mathcal{X})$ concides with the space BMO${_\mathcal{L}}(\mathbb{R}^n)$ and BLO${_\rho^{q}}(\mathcal{X})$ concides
with the space BLO${_\mathcal{L}}(\mathbb{R}^n)$. Meanwhile, for $\alpha\in\mathbb{R},\ p\in(0,+\infty),\ \mathcal{X}=\mathbb{R}^n,$ and$\ \mathcal{D}=D_\rho$, we denote $\mathcal{E}{_{\mathcal{D}}^{\alpha,p}}(\mathcal{X})$,\ $\widetilde{\mathcal{E}}{_{\mathcal{D}}^{\alpha,p}}(\mathcal{X})$ as $\mathcal{E}{_\mathcal{L}^{\alpha,p}}(\mathbb{R}^n)$, $\widetilde{\mathcal{E}}{_\mathcal{L}^{\alpha,p}}(\mathbb{R}^n)$, respectively.

We are now in a position to introduce the Localized Morrey-Campanato type spaces.
\begin{definition}(Localized Morrey-Campanato type spaces)
	Let $q\geq \frac{n}{2},\ V\in \mathcal{B}_q,\ 0<p <+\infty\ $, $ \alpha\in\mathbb{R}$ and  $f\in L{_{loc}^p}(\mathbb{R}^n)$.\\
		{\rm (i)} A function $f$ is said to be in the {\it Localized Morrey-Campanato space} $\mathcal{E}{_\mathcal{L}^{\alpha,p}}(\mathbb{R}^n)$ if
		\begin{align}\label{def:LMCS}
			\Vert f\Vert_{\mathcal{E}{_\mathcal{L}^{\alpha,p}}(\mathbb{R}^n)}:=\mathop{\sup}\limits_{B\notin D_\rho}\bigg(\frac{1}{\left|B\right|^{1+p\alpha}}\int_{B}\left|f(y)-f_B\right|^pdy\bigg)^\frac{1}{p}+\mathop{\sup}\limits_{B\in D_\rho}\frac{\left(\left|f\right|^p\right){_B^{\frac{1}{p}}}}{\left|B\right|^{\alpha}}<+\infty;
			\end{align}
		{\rm (ii)} A function $f$ is said to be in the {\it Localized Morrey-Campanato-BLO space} $\widetilde{\mathcal{E}}{_\mathcal{L}^{\alpha,p}}(\mathbb{R}^n)$ if 
		\begin{align}\label{def:LMCBS}
			\Vert f\Vert_{\widetilde{\mathcal{E}}{_\mathcal{L}^{\alpha,p}}(\mathbb{R}^n)}:=\mathop{\sup}\limits_{B\notin D_\rho}\bigg(\frac{1}{\left|B\right|^{1+p\alpha}}\int_{B}\big[f(y)-\mathop{{\rm essinf}}\limits_{B} f\big]^pdy\bigg)^\frac{1}{p}+\mathop{\sup}\limits_{B\in D_\rho}\frac{\left(\left|f\right|^p\right){_B^{\frac{1}{p}}}}{\left|B\right|^{\alpha}}<+\infty.
		\end{align}
\end{definition}
\noindent The following theorem shows that the Littlewood-Paley $g$-function associated with $\mathcal{L}$ is bounded from the localized Morrey-Campanato space to the localized Morrey-Campanato-BLO space on $\mathbb{R}^n$.

\hspace{-20pt}{\bf Theorem A} (\cite{YYZ3}).\label{YYZ3th} {\it\
	Let $q\in(n/2,\infty),\ V\in\mathcal{B}_q(\mathbb{R}^n,\left|\cdot\right|,dx),\ p\in(1,\infty)$ and let $\rho$ be the same as in $(\ref{eq:admissible})$. If $\alpha\in(-\infty,2/(3n)-1/(3q))$, then there exists a constant $C>0$ such that for any $f\in\mathcal{E}{_\mathcal{L}^{\alpha,p}}(\mathbb{R}^n)$, $g(f)\in\widetilde{\mathcal{E}}{_\mathcal{L}^{\alpha,p}}(\mathbb{R}^n)$ and 
	\begin{align*}
		\Vert g(f)\Vert_{\widetilde{\mathcal{E}}{_\mathcal{L}^{\alpha,p}}(\mathbb{R}^n)}\leq C\Vert f\Vert_{\mathcal{E}{_\mathcal{L}^{\alpha,p}}(\mathbb{R}^n)}.
	\end{align*}}

	%Due to the close correlation between the Littlewood-Paley g-function and Marcinkiewicz integral operator, it is necessary for us to give the definition of it. 
 On the other hand, in 1958, Stein \cite{S1}  introduced and studied the  following high dimensional Marcinkiewicz integral operator $\mu_{\Omega}$,
$$
\mu_{\Omega} (f)(x)=\bigg(\int_0^{\infty}\bigg|\int_{|x-y| \leq t} \frac{\Omega(x-y)}{|x-y|^{n-1}} f(y) d y\bigg|^2 \frac{d t}{t^3}\bigg)^{1 / 2}.
$$	
which can be regarded essentially as a special kind of the classical Littlewood-Paley $g$-functions. Here, $\mathbb{S}^{n-1}$ is the unit sphere in $\mathbb{R}^n$ equipped with normalized Lebesgue measure $d \sigma=d \sigma(x')$, $\Omega \in L^1\left(\mathbb{S}^{n-1}\right)$ is homogeneous of degree zero and satisfies
\begin{align}\label{tj:omegaxiaoshi}
	\int_{\mathbb{S}^{n-1}} \Omega(x') d \sigma(x')=0,
\end{align}
where $x'=\frac{x}{\left|x\right|}$ for any $x\neq0$. Since then, many nice works have been done, and it was shown in \cite{DFP,XCYFP} that  $\mu_\Omega$ is of type $(p, p)$ for $1<p<\infty$ provided that  $\Omega\in H^1(\mathbb{S}^{n-1})$  is homogeneous of degree zero, satisfying (\ref{tj:omegaxiaoshi}), and of  weak type (1, 1)  with $\Omega \in L\log^+ L({\mathbb{S}^{n-1}}) $ in \cite{FS} .  For more related works, we refer the
readers to \cite{BCP,DLXY,DLX1,DLX2,H2,HMY1,Q,SQCP,TW,W,Y} and the references therein.

Similar to the classical Marcinkiewicz integral operator $\mu_\Omega$, the Marcinkiewicz integral operator associated with the Schr\"{o}dinger operator $\mathcal{L}$ is defined by
\begin{align}\label{op:MO}
	\mu{_j^\mathcal{L}}(f)(x)=\bigg(\int_{0}^{\infty}\bigg|\int_{\left|x-y\right|\leq t}K{_j^\mathcal{L}}(x,y)f(y)dy\bigg|^2\frac{dt}{t^3}\bigg)^\frac{1}{2},
\end{align}
and the rough Marcinkiewicz integral operator associated with $\mathcal{L}$ is defined by
\begin{align}\label{op:KMO}
	\mu{_{j,\Omega}^\mathcal{L}}(f)(x)=\bigg(\int_{0}^{\infty}\bigg|\int_{\left|x-y\right|\leq t}\left|\Omega(x-y)\right|K{_j^\mathcal{L}}(x,y)f(y)dy\bigg|^2\frac{dt}{t^3}\bigg)^\frac{1}{2},
\end{align}
where $K{_j^\mathcal{L}}(x,y)=\widetilde{K{_j^\mathcal{L}}}(x,y)\left|x-y\right|$ and $\widetilde{K{_j^\mathcal{L}}}(x,y)$ is the kernel of $\textit{R}{_j^\mathcal{L}}=\frac{\partial}{\partial x_j}\mathcal{L}^{-\frac{1}{2}}, j=1,\cdots,n$. 
\begin{remark}
Let $V=0,\ K{_j^\Delta}(x,y)=\widetilde{K{_j^\Delta}}(x,y)\left|x-y\right|=\frac{(x_j-y_j)}{\left|x-y\right|^{n}}$ and $\widetilde{K{_j^\Delta}}(x,y)$ be the kernel of classical Riesz transform $\textit{R}{_j}=\frac{\partial}{\partial x_j}\Delta^{-\frac{1}{2}},\,j=1,\cdots,n$. Denote $K{_j}(x,y)=K{_j^\Delta}(x,y)$, then
\begin{align*}
	\mu{_j}(f)(x)=\bigg(\int_{0}^{\infty}\left|\int_{\left|x-y\right|\leq t}K{_j}(x,y)f(y)dy\right|^2\frac{dt}{t^3}\bigg)^\frac{1}{2},
\end{align*}
and
\begin{align}\label{eq:xmomega}
	\mu{_{j,\Omega}}(f)(x)=\bigg(\int_{0}^{\infty}\bigg|\int_{\left|x-y\right|\leq t}\left|\Omega(x-y)\right|K{_j}(x,y)f(y)dy\bigg|^2\frac{dt}{t^3}\bigg)^\frac{1}{2}.
\end{align} concide with the classical Marcinkiewicz integral type operators. \end{remark}

It was showed by  Gao and Tang in \cite{GT1}  that, for each $j$, $\mu{_j^\mathcal{L}}$ is bounded on $L^p(\mathbb{R}^n)$ for $1<p<\infty$ and bounded from $L^1(\mathbb{R}^n)$ to weak $L^1(\mathbb{R}^n)$. Meanwhile, Akbulut and Kuzu \cite{AK} demonstrated  the same conclusion for $\mu{_{j,\Omega}^\mathcal{L}}$ under the assumptions that $\Omega\in L^q(\mathbb{S}^{n-1})\,(1<q\leq\infty)$, is of  homogeneous of degree zero.  For more works, we refer the readers to see \cite{ZHX} and the references therein.

Comparing (\ref{def:gwithl}) with (\ref{op:MO}), it is easy to see that the kernels generated by $\mathcal{L}$ is different. This leads to a natural question whether Theorem {\bf A} holds for $\mu{_j^\mathcal{L}}$, or even $\mu{_{j,\Omega}^\mathcal{L}}$. In \cite{GT2}, Gao and Tang considered the case $\alpha=0$. It was shown that each $\mu{_j^\mathcal{L}}$ is bounded from BMO$_{\mathcal{L}}(\mathbb{R}^n)$ to BLO$_{\mathcal{L}}(\mathbb{R}^n)$, which left the case $\alpha\neq 0$ open. \vspace{0.2cm}

\hspace{-20pt}{\bf Theorem B} (\cite{GT2}).\label{GT2th}{\it\ \
	Let $V\in\mathcal{B}_n$. The Marcinkiewicz integral operators $\mu{_j^\mathcal{L}}\,(j=1,2,\cdots,n)$ associated with the Schr\"{o}dinger operator are bounded from BMO$_{\mathcal{L}}(\mathbb{R}^n)$ to BLO$_{\mathcal{L}}(\mathbb{R}^n)$. Furthermore, there exists a constant $C>0$ such that for any $f\in$ BMO$_{\mathcal{L}}$\,,
	\begin{align*}
		\Vert\mu{_j^\mathcal{L}}(f)\Vert_{\BLO_{\mathcal{L}}}\leq C\Vert f\Vert_{\BMO_{\mathcal{L}}}.
	\end{align*}
}
The first purpose of this paper is to give a positive answer to the above question, and thus extend Theorem {\bf B}, the case for $\alpha=0$ to the general case $\alpha\neq 0$ for $\mu{_j^\mathcal{L}}$. Our first main result of this paper is as follows:

\begin{theorem}\label{th1}
	Let $V\in\mathcal{B}_n\ p\in (1,+\infty),\ 0<\delta<2-\frac{n}{q_0},\ 0<\gamma<1-\frac{n}{q_0},\ 0<\beta<\min \big\{1,1-\frac{n}{q_0}\big\}.$ If $\alpha\in\left(-\infty,\frac 1n\min\big\{{\delta},{1-2\delta},\frac{\gamma}{2},\frac{\beta}{2},\frac{1}{3}\big\}\right)$, then for each $j,$ $1\le j\le n$, $\mu{_j^\mathcal{L}}$ is bounded from $\mathcal{E}{_\mathcal{L}^{\alpha,p}}(\mathbb{R}^n)$ to $\widetilde{\mathcal{E}}{_\mathcal{L}^{\alpha,p}}(\mathbb{R}^n)$. Furthermore, there exists a constant $C>0$ such that for any $f\in$ $\mathcal{E}{_\mathcal{L}^{\alpha,p}}(\mathbb{R}^n)$,
	\begin{align*}
		\Vert\mu{_j^\mathcal{L}}(f)\Vert_{\widetilde{\mathcal{E}}{_\mathcal{L}^{\alpha,p}}(\mathbb{R}^n)}\leq C\Vert f\Vert_{\mathcal{E}{_\mathcal{L}^{\alpha,p}}(\mathbb{R}^n)}.
	\end{align*}
\end{theorem}
Theorem \ref {th1} can be extended to the Marcinkiewicz integral operators $\mu{_{j,\Omega}^\mathcal{L}}$ with rough kernels and there is no need to assume the vanishing condition on $\Omega$. There are major differences between the kernels and these extensions are  certainly non trival and have interests of their own.
To begin with, we need the definition of $L^q$-Dini type condition for the kernel.

Let $\Omega\in L^q(\mathbb{S}^{n-1})\,(q>1)$ and $\omega_q$ be the integral modulus of continuity of order $q$ of $\Omega$, which is defined by $\omega_q(x)=\sup\limits_{|\rho|\leq\sigma}\left(\int_{\mathbb{S}^{n-1}}\left|\Omega(\rho x)-\Omega(x)\right|^qdx\right)^{{1}/{q}}$, where $\rho$ is a rotation on $\mathbb{S}^{n-1}$. Then $\Omega(x)$ is said to satisfy the $L^q$-Dini type condition if  it holds that
\begin{align}\label{def:lqDini2}
\int_{0}^{1}\frac{\omega_q(\delta)}{\delta^{1+\epsilon}}d\delta<\infty, \quad\hbox{\ for some\  }0<\epsilon\leq1.
\end{align}
We now state our second main result as follows:

\begin{theorem}\label{th3}
	Let $1<p<\infty$ and $\Omega\in L^q(\mathbb{S}^{n-1})\,(\ 1\leq q'\leq p~)$ be homogeneous of degree zero. Suppose $V\in\mathcal{B}_n,$ the range of $\delta,\ \gamma$ and $\beta$ are the same as in Theorem \ref{th1}. Then, for $1\le j\le n$, it holds that \\
\begin{enumerate} \item[(i)] if $\alpha\leq0$, then each $\mu{_{j,\Omega}^\mathcal{L}}$ is bounded from $\mathcal{E}{_\mathcal{L}^{\alpha,p}}(\mathbb{R}^n)$ to $\widetilde{\mathcal{E}}{_\mathcal{L}^{\alpha,p}}(\mathbb{R}^n)$;\\
	\item[(ii)] if \ $0<\alpha<\frac 1n \min\big\{{\delta}, {1-2\delta}, \frac{\gamma}{2}, \frac{\beta}{2}, \frac{\epsilon}{3}\big\}$ and $\Omega$ satisfies the Dini type condition 
(\ref{def:lqDini2}). Then each $\mu{_{j,\Omega}^\mathcal{L}}$ is bounded from $\mathcal{E}{_\mathcal{L}^{\alpha,p}}(\mathbb{R}^n)$ to $\widetilde{\mathcal{E}}{_\mathcal{L}^{\alpha,p}}(\mathbb{R}^n)$ and for any $f\in$ $\mathcal{E}{_\mathcal{L}^{\alpha,p}}(\mathbb{R}^n)$ the following norm inequality holds
\begin{align*}
\Vert\mu{_j^\mathcal{L}}(f)\Vert_{\widetilde{\mathcal{E}}{_\mathcal{L}^{\alpha,p}}(\mathbb{R}^n)}\leq C\Vert f\Vert_{\mathcal{E}{_\mathcal{L}^{\alpha,p}}(\mathbb{R}^n)}.
\end{align*}\end{enumerate}
\end{theorem}

As a consequence of Theorem \ref {th3}, the following corollary holds when $\alpha=0,$ which essentially extends Theorem {\bf {B}} to the case of rough kernels.
\begin{corollary}\label{th4}
	Let $1<p<\infty$ and $\Omega\in L^q(\mathbb{S}^{n-1})\,(\ 1\leq q'\leq p~)$ be homogeneous of degree zero. Suppose $V\in\mathcal{B}_n,$ then for $1\le j\le n$,  each $\mu{_{j,\Omega}^\mathcal{L}}$ is bounded from BMO$_{\mathcal{L}}(\mathbb{R}^n)$ to BLO$_{\mathcal{L}}(\mathbb{R}^n)$.
\end{corollary}

	This paper is organized as follow. In Section \ref{section2}, we give some notions and lemmas, which play a fundamental role in our analysis. Section \ref{section4} and Section \ref{section5} are devoted to prove our main theorems. Throughout the paper, if there exists a constant $C>0$, independent with the essential variables but not necessarily the same one at each occurrence, such that $A\leq CB$, we then write $A\lesssim B$. By $A\backsim B$, it means that $A$ is equivalent to $B$.

\section{Preliminaries}\label{section2}
In this section, we will present some useful lemmas. The first one is about  the properties of the auxiliary function $\rho$.
\begin{lemma}{\rm(\cite{S3})}\label{lem:fuzhuhanshu}
	Let $V\in\mathcal{B}_n.$  Then there exist $k_0,\,C,\,c>0$, such that for any $x,\,y\in\mathbb{R}^n$, 
	\begin{align*}
		c\left(1+\frac{\left|x-y\right|}{\rho(x)}\right)^{-\frac{k_0}{k_0+1}}\leq \frac{\rho(y)}{\rho(x)}\leq C\left(1+\frac{\left|x-y\right|}{\rho(x)}\right)^{k_0}.
	\end{align*}
	In particular, $\rho(x)\backsim\rho(y)$ if $\left|x-y\right|<C\rho(x)$.
\end{lemma}

	Let $l_0=\frac{k_0}{k_0+1}$. Then $0<l_0<1$. This will be used in the proof of Theorem \ref{th1}- \ref{th3}.

We need to present some known estimates for the kernels of Marcinkiewicz integral operators associated with the Riesz transform $\nabla\mathcal{L}^{-\frac{1}{2}}$.
\begin{lemma}{\rm(\cite{SM})}\label{lem:kguji}
There exist $C_1,\,C_2>0$, such that
	\begin{enumerate}
		\item[\rm(i)]$\left|K{_j}(x,y)\right|\leq\frac{C_1}{\left|x-y\right|^{n-1}};$
		\item[\rm(ii)]$\left|K{_j}(x+h,y)-K{_j}(x,y)\right|\leq\frac{C_2\left|h\right|}{\left|x-y\right|^{n}};$
		\item[\rm(iii)]$\int_{\left|y-x\right|=1}K{_j}(x,y)dx=0.$
	\end{enumerate}
\end{lemma}

\begin{lemma}{\rm(\cite{S3})}\label{lem:kjlguji}
	Let $V\in\mathcal{B}_n.$ Then
	\begin{enumerate}
		\item[\rm(i)] for any $l>0$, there exists $C_l>0$ such that
		\begin{align*}
			\left|K{_j^\mathcal{L}}(x,y)\right|\leq C_l\left(1+\frac{\left|x-y\right|}{\rho(x)}\right)^{-l}\frac{1}{\left|x-y\right|^{n-1}}.
		\end{align*}
		\item[\rm(ii)] for every $N>0$ and $0<\beta<\min \{1,1-\frac{n}{q_0}\}$, there exists a constant $C>0$ such that
		\begin{align*}
			\left|K_j^\mathcal{L}(x, z)-K_j^\mathcal{L}(y, z)\right| \leq \frac{C|x-y|^\beta(1+(|x-z| / \rho(x)))^{-N}}{|x-z|^{n-1+\beta}},
		\end{align*}
		whenever $|x-y|<\frac{2}{3}|x-z|$.\\
		\item[\rm(iii)]
		for every $0<\delta<2-\frac{n}{q_0}$, there exists a constant $C>0$ such that
		\begin{align*}
			\left|K_j^\mathcal{L}(x, z)-K_j(x, z)\right| \leq \frac{C}{|x-z|^{n-1}}\left(\frac{|x-z|}{\rho(z)}\right)^\delta.
		\end{align*}
	\end{enumerate}
\end{lemma}

\begin{lemma}{\rm(\cite{BHS})}\label{lem:kjlguji4} 
	Let $V\in\mathcal{B}_q$ with $q>n$ and $0<\gamma<1-\frac{n}{q}$. Then there exists a constant $C>0$ such that
\begin{align*}
	|[\widetilde{K_j^\mathcal{L}}(x, z)-\widetilde{K_j^\Delta}(x, z)]-[\widetilde{K_j^\mathcal{L}}(y, z)-\widetilde{K_j^\Delta}(y, z)]| \leq \frac{C|x-y|^\gamma}{|x-z|^{n+\gamma}}\left(\frac{|x-z|}{\rho(x)}\right)^{\eta'},
\end{align*}
	whenever $|x-z| \geq 2|x-y|$ and $\eta'=2-n / q$.
\end{lemma}

\begin{remark}\label{lem:kjlcha1}
 In Lemma \ref{lem:kjlguji4}, $\eta'$ can be replaced by any $\eta_0$ satisfying $0<\eta_0\leq\eta'$. Indeed, by the proof of Lemma 4 in \cite{BHS}, since both kernels of $\widetilde{K{_j^\mathcal{L}}}(x,y)$ and $\widetilde{K{_j^\Delta}}(x,y)$ satisfy the Calder$\acute{o}$n-Zygmund smoothness estimate with $0<\gamma<1-\frac{n}{q}$, the following inequality holds in the case $|x-z|\geq\rho(x)$,
	$$
	|[\widetilde{K{_j^\mathcal{L}}}(x,z)-\widetilde{K{_j^\Delta}}(x,z)]-[\widetilde{K{_j^\mathcal{L}}}(y,z)-\widetilde{K{_j^\Delta}}(y,z)]| \leq \frac{C|x-y|^\gamma}{|x-z|^{n+\gamma}}\leq \frac{C|x-y|^\gamma}{|x-z|^{n+\gamma}}\left(\frac{|x-z|}{\rho(x)}\right)^{\eta_0}.
	$$ 
	When $|x-z|<\rho(x)$,  it follows from the estimate in \cite{BHS} and $0<\eta_0\leq\eta'$ that 
	\begin{align*}
		|[\widetilde{K{_j^\mathcal{L}}}(x,z)-\widetilde{K{_j^\Delta}}(x,z)]-[\widetilde{K{_j^\mathcal{L}}}(y,z)-\widetilde{K{_j^\Delta}}(y,z)]|
		&\leq\frac{C|x-y|^\gamma}{|x-z|^{n+\gamma}}\left(\frac{|x-z|}{\rho(x)}\right)^{\eta'}\\
		&\leq\frac{C|x-y|^\gamma}{|x-z|^{n+\gamma}}\left(\frac{|x-z|}{\rho(x)}\right)^{\eta_0}.
	\end{align*}
	Therefore, for any $0<\eta_0\leq\eta'$, 
	$$
	|[\widetilde{K{_j^\mathcal{L}}}(x,z)-\widetilde{K{_j^\Delta}}(x,z)]-[\widetilde{K{_j^\mathcal{L}}}(y,z)-\widetilde{K{_j^\Delta}}(y,z)]|\leq\frac{C|x-y|^\gamma}{|x-z|^{n+\gamma}}\left(\frac{|x-z|}{\rho(x)}\right)^{\eta_0}.
	$$
\end{remark} 

\begin{remark}\label{lem:kjlcha2}
	When $V \in\mathcal{B}_n$, there exists $q_0>n$ such that $V\in\mathcal{B}_{q_0}$. We may use Lemma \ref{lem:kjlguji4} and Remark \ref{lem:kjlcha1} to conclude that there exists a constant $C>0$  such that 
	$$
	|[\widetilde{K_j^\mathcal{L}}(x, z)-\widetilde{K_j^\Delta}(x, z)]-[\widetilde{K_j^\mathcal{L}}(y, z)-\widetilde{K_j^\Delta}(y, z)]| \leq \frac{C|x-y|^\gamma}{|x-z|^{n+\gamma}}\left(\frac{|x-z|}{\rho(x)}\right)^{\eta}, \ \hbox{ for \ }0<\gamma<1-\frac{n}{q_0},
	$$
	where $|x-z| \geq 2|x-y|,\,\eta=1 / 2-n / (2q_0)$. Here $0<\eta-\gamma/ 2<1.$ This estimate plays an important role in our analysis.
\end{remark} 

We are now in a position to state the $L^p$ boundedness of $\mu{_j^\mathcal{L}}$ and $\mu{_{j,\Omega}^\mathcal{L}}$.
\begin{lemma}{\rm(\cite{GT1})}\label{lem:kjlpp}
	Let $V\in\mathcal{B}_n$. Then operators $\mu{_j^\mathcal{L}}\,(j=1,2,\cdots,n)$ are bounded on $L^p(\mathbb{R}^n)$ for $1<p<\infty$, and bounded from $L^1(\mathbb{R}^n)$ to weak $L^1(\mathbb{R}^n)$.
\end{lemma}

\begin{lemma}{\rm}\label{lem:kjlOmegapp} 
	Let $V\in\mathcal{B}_n,\ 1<p<\infty$ and $\Omega\in L^q(\mathbb{S}^{n-1})\,(1<q\leq\infty)$ be homogeneous of degree zero, then the operators $\mu{_{j,\Omega}^\mathcal{L}}\,(j=1,2,\cdots,n)$ are bounded from $L^p(\mathbb{R}^n)$ for $1<p<\infty$, and bounded from $L^1(\mathbb{R}^n)$ to weak $L^1(\mathbb{R}^n)$.
\end{lemma}
\begin{proof}Note that $\Omega$ is homogeneous of degree zero. Then we have
	\begin{align*}
	\int_{\left|x-y\right|\leq1} \left|\Omega(x-y)\right| k_j(x,y)d \sigma=0.
	\end{align*}
	Using the ideas from \cite{AK,GT1}, it follows that
	\begin{align*}
		\mu{_{j,\Omega}^\mathcal{L}}f(x)\leq\mu{_{j,\Omega}}f(x)+cM_{\Omega}f(x),\ a.e.\,x\in\mathbb{R}^n,
	\end{align*}
	where $M_{\Omega}$ is the classical maximal operator with homogeneous kernel defined by 
	$$M_\Omega f(x)=\sup\limits_{t>0}\frac{1}{t^n}\int_{\left|x-y\right|\leq t}\left|\Omega(x-y)f(y)\right|dy.$$

	Therefore, $\mu_{j,\Omega}$ is  the classical Marcinkiewicz integral operators with a kernel of homogeneous of degree zero, satisfying mean value property on $\mathbb{S}^{n-1}$.  Lemma \ref {lem:kjlOmegapp}  then follows from the $L^p\ (1<p<\infty)$ and weak $(1,1)$ boundedness of classical Marcinkiewicz integral \cite{DFP,FS} and the properties of the  classical maximal operator \cite{SM} (p. 72) and  \cite{CH}.
\end{proof}

The functions in Localized Morrey-Campanato type spaces enjoy the following properties.
\begin{lemma}{\rm(\cite{YYZ3})}\label{lem:fcampanato1} 
	Let $\alpha \in \mathbb{R},\,p \in[1, \infty)$. Then there exists a positive constant $C$ such that for all $f \in \mathcal{E}{_\mathcal{L}^{\alpha,p}}(\mathbb{R}^n)$,
		\begin{enumerate}
		\item[\rm(i)]
for all balls $B \equiv B\left(x_0, r\right) \notin D_\rho$,
	$$
	\frac{1}{\left|B\right|} \int_B|f(z)| d z \leq \begin{cases}C\left(\frac{\rho\left(x_0\right)}{r}\right)^{\alpha n}\left|B\right|^\alpha\|f\|_{\mathcal{E}{_\mathcal{L}^{\alpha,p}}(\mathbb{R}^n)}, & \alpha>0, \\ C\left(1+\log_2 \frac{\rho\left(x_0\right)}{r}\right)\left|B\right|^\alpha\|f\|_{\mathcal{E}{_\mathcal{L}^{\alpha,p}}(\mathbb{R}^n)}, & \alpha \leq 0 ;\end{cases}
	$$
\item[\rm(ii)]for all $x \in \mathbb{R}^n$ and $0<r_1<r_2$,
	$$
	\left|f_{B\left(x, r_1\right)}-f_{B\left(x, r_2\right)}\right| \leq \begin{cases}C\left(\frac{r_2}{r_1}\right)^{\alpha n}\left|B(x,r_1)\right|^\alpha\|f\|_{\mathcal{E}{_\mathcal{L}^{\alpha,p}}(\mathbb{R}^n)}, & \alpha>0, \\ C\left(1+\log_2 \frac{r_2}{r_1}\right)\left|B(x,r_1)\right|^\alpha\|f\|_{\mathcal{E}{_\mathcal{L}^{\alpha,p}}(\mathbb{R}^n)}, & \alpha \leq 0 .\end{cases}
	$$	\end{enumerate}
\end{lemma} 

\begin{lemma}\label{lemm:fcampanato2}
	Let $\alpha \in \mathbb{R},\,p \in[1, \infty)$. Then there exists a positive constant $C$ such that for all $f \in \mathcal{E}{_\mathcal{L}^{\alpha,p}}(\mathbb{R}^n)$ and for all balls $B \equiv B\left(x_0, r\right) \notin D_\rho$,
	\begin{align}\label{lem:fcampanato2}
		\left(\frac{1}{\left|B\right|} \int_B|f(z)|^p d z\right)^{\frac{1}{p}} \leq \begin{cases}C\left(\frac{\rho\left(x_0\right)}{r}\right)^{\alpha n}\left|B\right|^\alpha\|f\|_{\mathcal{E}{_\mathcal{L}^{\alpha,p}}(\mathbb{R}^n)}, & \alpha>0, \\ C\left(1+\log_2 \frac{\rho\left(x_0\right)}{r}\right)\left|B\right|^\alpha\|f\|_{\mathcal{E}{_\mathcal{L}^{\alpha,p}}(\mathbb{R}^n)}, & \alpha \leq 0.
		\end{cases}
	\end{align}
\end{lemma}
\begin{proof}
	For all $f \in \mathcal{E}{_\mathcal{L}^{\alpha,p}}(\mathbb{R}^n)$ and balls $B \equiv B\left(x_0, r\right) \notin D_\rho$, the Minkowski inequality gives
	\begin{align*}
		\left(\frac{1}{\left|B\right|} \int_B|f(z)|^p d z\right)^{\frac{1}{p}}
		&\leq\left(\frac{1}{\left|B\right|} \int_B|f(z)-f_B|^p d z\right)^{\frac{1}{p}}+\left|f_B-f_{B(x_0,\rho(x_0))}\right|+|f_{B(x_0,\rho(x_0))}|\\
		&\lesssim\left|B\right|^\alpha+\left|f_B-f_{B(x_0,\rho(x_0))}\right|+\left(\frac{\rho\left(x_0\right)}{r}\right)^{\alpha n}\left|B\right|^\alpha.
	\end{align*}
	By Lemma \ref{lem:fcampanato1}~(ii), it follows that
	\begin{align*}
		\left|f_B-f_{B(x_0,\rho(x_0))}\right|\lesssim \begin{cases}\left(\frac{\rho\left(x_0\right)}{r}\right)^{\alpha n}\left|B\right|^\alpha, & \alpha>0, \\ \left(1+\log_2 \frac{\rho\left(x_0\right)}{r}\right)\left|B\right|^\alpha, & \alpha \leq 0 ;\end{cases}
	\end{align*}
	Thus (\ref{lem:fcampanato2}) is proved.
\end{proof}

We will use the following well-known lemma related to the $L^q$-Dini condition.
\begin{lemma}{\rm(\cite{KW})}\label{lm:omegalqdini}
	Suppose that $0<\lambda<n$,and $\Omega$ is homogeneous of degree zero and satisfies the $L^q$-Dini type condition with $\epsilon=0$ for $q>1$. If  there exists a constant $a_0>0$  such that $|x|<a_0 R$, then we have 
	$$
	\begin{aligned}
		\left(\int_{R<|y|<2 R}\left|\frac{\Omega(y-x)}{|y-x|^{n-\lambda}}-\frac{\Omega(y)}{|y|^{n-\lambda}}\right|^q d y\right)^{1 / q}\leq C R^{n / q-(n-\lambda)}\left\{\frac{|x|}{R}+\int_{|x| / 2 R<\delta<|x| / R} \frac{\omega_q(\delta)}{\delta} d \delta\right\},
	\end{aligned}
	$$
	where the constant $C>0$ is independent of $R$ and $x$.
\end{lemma}
\begin{remark}[\cite{DLX2}]
	It is obvious that the $L^q$-Dini type condition with $\epsilon=0$ is weaker than the condition (\ref{def:lqDini2}).
\end{remark}

In order to prove our theorems, we begin with some main lemmas about the inner layer integral estimates for $\mu{_j^\mathcal{L}}$ and $\mu{_{j,\Omega}^\mathcal{L}}$.

\begin{lemma}\label{lem:kernelguji}
	Let $\alpha\in\mathbb{R},\ p\in(1,+\infty).$ For any $x\in\mathbb{R}^n,\ t>0$ and $0\leq s<\min\{1,\,1+n\alpha\}$, there exists a positive constant such that for any $f\in\mathcal{E}{_\mathcal{L}^{\alpha,p}}(\mathbb{R}^n),$ 
	\begin{align*}
		\bigg|\int_{\left|x-z\right|\leq t}K{_j^\mathcal{L}}(x,z)f(z)dz\bigg|\lesssim t\bigg(\frac{\rho(x)}{t}\bigg)^s\left|B(x,t)\right|^\alpha.
	\end{align*}
\end{lemma}

\begin{proof}
	Let  $f\in\mathcal{E}{_\mathcal{L}^{\alpha,p}}(\mathbb{R}^n).$ By the homogeneity of $\Vert \cdot\Vert_{\mathcal{E}{_\mathcal{L}^{\alpha,p}}(\mathbb{R}^n)}$, we may assume $\Vert f\Vert_{\mathcal{E}{_\mathcal{L}^{\alpha,p}}(\mathbb{R}^n)}=1.$ We need to consider two cases according to the relationship between $t$ and $\rho(x)$.
	
	\textbf{Case (i)}: $t\geq\rho(x)$. In this case, Lemma~\ref{lem:kjlguji}~(i) yields
	\begin{align*}
		\Big|\int_{\left|x-z\right|\leq t}K{_j^\mathcal{L}}(x,z)f(z)dz\Big|\lesssim&\left[\rho(x)\right]^s\int_{\left|x-z\right|\leq t}\frac{1}{\left|x-z\right|^{n+s-1}}\left|f(z)\right|dz\\
		\lesssim&\left[\rho(x)\right]^s\sum_{k=-\infty}^{0}\frac{(2^kt)^{1-s}}{\left|B(x,2^kt)\right|}\int_{\left|x-z\right|\leq 2^kt}\left|f(z)\right|dz.
	\end{align*}
	Since $t\geq\rho(x)$, there  exists an integer $k_0\in(-\infty,0]$ such that $2^{k_0-1}t<\rho(x)\leq2^{k_0}t$. \\
Consider first the case $k\in(-\infty,k_0]$. It follows that $2^kt\geq2^{k_0}t\geq\rho(x).$ By the H\"{o}lder inequality and (\ref{def:LMCS}), 
	\begin{align*}%\label{eq:3}
		\frac{1}{\left|B(x,2^kt)\right|}\int_{\left|x-z\right|\leq 2^kt}\left|f(z)\right|dz\leq\Big(\frac{1}{\left|B(x,2^kt)\right|}\int_{\left|x-z\right|\leq 2^kt}\left|f(z)\right|^pdz\Big)^{\frac{1}{p}}\lesssim\left|B(x,2^kt)\right|^\alpha.
	\end{align*}
	Therefore
	\begin{align}\label{eq:1}
		\sum_{k=-\infty}^{k_0}\frac{(2^kt)^{1-s}}{\left|B(x,2^kt)\right|}\int_{\left|x-z\right|\leq 2^kt}\left|f(z)\right|dz]
		\lesssim t^{1-s}\left|B(x,t)\right|^\alpha.
	\end{align}
For the case $k\in(k_0,0]$, since $2^kt\leq2^{k_0-1}t<\rho(x).$ Then Lemma \ref{lem:fcampanato1}~(i) gives
	\begin{align*}
		\frac{1}{\left|B(x,2^kt)\right|}\int_{\left|x-z\right|\leq 2^kt}\left|f(z)\right|dz
		&\lesssim\begin{cases}\left(2^{k_0-k}\right)^{\alpha n}\left|B(x,2^kt)\right|^\alpha, & \alpha>0, \\ \left(1-k+k_0\right)\left|B(x,2^kt)\right|^\alpha, & \alpha \leq 0 .\end{cases}
	\end{align*}
	For $\alpha>0$, we have for all $0\leq s<1$,
	\begin{align*}
		\sum_{k=k_0+1}^{0}\frac{(2^kt)^{1-s}}{\left|B(x,2^kt)\right|}\int_{\left|x-z\right|\leq 2^kt}\left|f(z)\right|dz
		\lesssim t^{1-s}\left|B(x,t)\right|^\alpha.
	\end{align*}
	For $\alpha\leq0$, since $\min\{1,\,1+n\alpha\}=1+n\alpha$ and $1+k_0-k\leq1$, then for all $0\leq s<1+n\alpha$,
	\begin{align*}
		\sum_{k=k_0+1}^{0}\frac{(2^kt)^{1-s}}{\left|B(x,2^kt)\right|}\int_{\left|x-z\right|\leq 2^kt}\left|f(z)\right|dz\lesssim t^{1-s}\left|B(x,t)\right|^\alpha.
	\end{align*}
	Therefore, for all $\alpha\in\mathbb{R}$,
	\begin{align}\label{eq:2}
		\sum_{k=k_0+1}^{0}\frac{(2^kt)^{1-s}}{\left|B(x,2^kt)\right|}\int_{\left|x-z\right|\leq 2^kt}\left|f(z)\right|dz\lesssim t^{1-s}\left|B(x,t)\right|^\alpha.
	\end{align}
	Combining (\ref{eq:1}) and (\ref{eq:2}), we complete the case for $t\geq\rho(x)$.\\
	\textbf{Case (ii)}: $t<\rho(x)$. In this case, for all ~$0\leq s<\min\{1,\,1+n\alpha\}$, Lemma \ref{lem:kjlguji}~(i) gives 
	\begin{align*}
		&\bigg|\int_{\left|x-z\right|\leq t}K{_j^\mathcal{L}}(x,z)\left[f(z)-f_{B(x,t)}\right]dz\bigg|\\&
		\lesssim\left[\rho(x)\right]^s\int_{\left|x-z\right|\leq t}\frac{1}{\left|x-z\right|^{n+s-1}}\left|f(z)-f_{B(x,t)}\right|dz\\&
		\lesssim\left[\rho(x)\right]^s\sum_{k=-\infty}^{0}(2^kt)^{1-s}\Big\{\frac{1}{\left|B(x,2^kt)\right|}\int_{\left|x-z\right|\leq 2^kt}\big|f(z)-f_{B(x,2^kt)}\big|dz+\big|f_{B(x,2^kt)}-f_{B(x,t)}\big|\Big\}.
	\end{align*}
Note that by Lemma \ref{lem:fcampanato1}~(ii), we may obtain
	\begin{align*}
		\left|f_{B(x,2^kt)}-f_{B(x,t)}\right|&\lesssim
		\begin{cases}(2^{-k})^{n\alpha}\left|B(x,2^kt)\right|^\alpha, & \alpha>0, \\ (1-k)\left|B(x,2^kt)\right|^\alpha, & \alpha \leq 0 .
		\end{cases}
	\end{align*}
	Therefore for all $0\leq s<\min\{1,1+n\alpha\}$, $$
	\sum_{k=-\infty}^{0}(2^kt)^{1-s}\left|f_{B(x,2^kt)}-f_{B(x,t)}\right|
	\lesssim t^{1-s}\left|B(x,t)\right|^\alpha.$$
	Combining this inequality with (\ref{def:LMCS}) gives
	\begin{align}\label{eq:4}
		\bigg|\int_{\left|x-z\right|\leq t}K{_j^\mathcal{L}}(x,z)\left(f(z)-f_{B(x,t)}\right)dz\bigg|
		\lesssim t\big(\frac{\rho(x)}{t}\big)^s\left|B(x,t)\right|^\alpha.
	\end{align}
	On the other hand, by Lemma \ref{lem:fuzhuhanshu}, when $\left|x-z\right|\leq t<\rho(x)$, $\rho(x)\backsim\rho(z)$. Applying this with  $\mu_j(1)=0$, Lemma \ref{lem:kjlguji}~(iii), Lemma \ref{lem:fcampanato1} (i), and taking $\delta=\max\{1,\,\max\{0,n\alpha\}\}$, we obtain
	\begin{align}\label{eq:5}
		\bigg|\int_{\left|x-z\right|\leq t}K{_j^\mathcal{L}}(x,z)f_{B(x,t)}dz\bigg|
		\leq&\bigg|\int_{\left|x-z\right|\leq t}\left[K{_j^\mathcal{L}}(x,z)-K_j(x,z)\right]f_{B(x,t)}dz\bigg|\nonumber\\
		\lesssim&t\left[\frac{t}{\rho(x)}\right]^{\delta}\left|B(x,t)\right|^\alpha\begin{cases}\left(\frac{\rho\left(x\right)}{t}\right)^{n\alpha}, & \alpha>0, \\ 1+\log_2 \frac{\rho\left(x\right)}{t}, & \alpha \leq 0 ;
		\end{cases}\nonumber\\
		\lesssim&t\left[\frac{\rho(x)}{t}\right]^s\left|B(x,t)\right|^\alpha.
	\end{align}
	Summing up (\ref{eq:4}) and (\ref{eq:5}), we obtain the desired conclusion for $t<\rho(x)$, which completes the proof of Lemma \ref{lem:kernelguji}.
\end{proof}

\begin{lemma}\label{lem:kernelguji2}  Let $\alpha\in\mathbb{R},\ p\in(1,+\infty). $ Suppose $\Omega\in L^q(\mathbb{S}^{n-1})\,(\ 1\leq q'\leq p~)$ is homogeneous of degree zero. For any $x\in\mathbb{R}^n,\ t>0$, $0\leq s<\min\{1,\,1+n\alpha\}$, and any $f\in\mathcal{E}{_\mathcal{L}^{\alpha,p}}(\mathbb{R}^n),$
	\begin{align*}
\bigg|\int_{\left|x-z\right|\leq t}\left|\Omega(x-z)\right|K{_j^\mathcal{L}}(x,z)f(z)dz\bigg|\lesssim t\left[\frac{\rho(x)}{t}\right]^s\left|B(x,t)\right|^\alpha.
	\end{align*}
\end{lemma}

\begin{proof}
Let $f\in\mathcal{E}{_\mathcal{L}^{\alpha,p}}(\mathbb{R}^n)$ and assume $\Vert f\Vert_{\mathcal{E}{_\mathcal{L}^{\alpha,p}}(\mathbb{R}^n)}=1$ by homogeneity. Consider two cases.\\
	\textbf{Case (i)}: $t\geq\rho(x)$. By Lemma~\ref{lem:kjlguji}~(i) and the H\"{o}lder inequality, it yields that
	\begin{align*}
{}&	\bigg|\int_{\left|x-z\right|\leq t}\left|\Omega(x-z)\right|K{_j^\mathcal{L}}(x,z)f(z)dz\bigg|\\
		&\lesssim\left[\rho(x)\right]^s\sum_{k=-\infty}^{0}(2^kt)^{1-s}\Big(\frac{1}{\left|B(x,2^kt)\right|}\int_{\left|x-z\right|\leq 2^kt}\left|f(z)\right|^{p}dz\Big)^{\frac{1}{p}}.
	\end{align*}
	Similarly, for $t\geq\rho(x)$, their exists an integer $k_0\in(-\infty,0]$ such that $2^{k_0-1}t<\rho(x)\leq2^{k_0}t$.\\
For $k\in(-\infty,k_0]$, it follows that $2^kt\geq2^{k_0}t\geq\rho(x).$ By (\ref{def:LMCS}), we get
	\begin{align}\label{eq:6}
		\sum_{k=-\infty}^{0}(2^kt)^{1-s}\Big(\frac{1}{\left|B(x,2^kt)\right|}\int_{\left|x-z\right|\leq 2^kt}\left|f(z)\right|^{p}dz\Big)^{\frac{1}{p}}
			\lesssim& t^{1-s}\left|B(x,t)\right|^\alpha;
	\end{align}
For $k\in(k_0,0]$, it follows that $2^kt\leq2^{k_0-1}t<\rho(x).$  Lemma \ref{lemm:fcampanato2} then yields
	\begin{align*}
		\Big(\frac{1}{\left|B(x,2^kt)\right|}\int_{\left|x-z\right|\leq 2^kt}\left|f(z)\right|^{p}dz\Big)^{\frac{1}{p}}\lesssim\begin{cases}\left(2^{k_0-k}\right)^{\alpha n}\left|B(x,2^kt)\right|^\alpha, & \alpha>0, \\ \left(1-k+k_0\right)\left|B(x,2^kt)\right|^\alpha, & \alpha \leq 0 .\end{cases}
	\end{align*}
	Similarly as in Lemma \ref{lem:kernelguji},  for $k\in(k_0,0]$ and all $\alpha\in\mathbb{R}$, it holds that
	\begin{align}\label{eq:7}
		\sum_{k=k_0+1}^{0}(2^kt)^{1-s}\Big(\frac{1}{\left|B(x,2^kt)\right|}\int_{\left|x-z\right|\leq 2^kt}\left|f(z)\right|^{p}dz\Big)^{\frac{1}{p}}\lesssim t^{1-s}\left|B(x,t)\right|^\alpha.
	\end{align}
	Hence, the desired conclusion for $t\geq\rho(x)$ follows from (\ref{eq:6}) and (\ref{eq:7}) immediately.\\
	\textbf{Case (ii)}: $t<\rho(x)$. For all ~$0\leq s<\min\{1,\,1+n\alpha\}$,  applying Lemma \ref{lem:kjlguji}~(i), together with the H\"{o}lder inequality and the Minkowski inequality gives
	\begin{align*}
	{}	&\left|\int_{\left|x-z\right|\leq t}\left|\Omega(x-z)\right|K{_j^\mathcal{L}}(x,z)\left[f(z)-f_{B(x,t)}\right]dz\right|\\&
		\lesssim\left[\rho(x)\right]^s\sum_{k=-\infty}^{0}(2^kt)^{1-s}\Big\{\Big(\frac{1}{\left|B(x,2^kt)\right|}\int_{\left|x-z\right|\leq 2^kt}\left|f(z)-f_{B(x,2^kt)}\right|^pdz\Big)^{\frac{1}{p}}\\&\quad+\left|f_{B(x,2^kt)}-f_{B(x,t)}\right|\Big\}
	\end{align*}
Using Lemma \ref{lem:fcampanato1}~(ii), the right side of the above inequality is bounded by
	\begin{align}\label{eq:8}
		t\left[\frac{\rho(x)}{t}\right]^s\sum_{k=-\infty}^{0}\left|B(x,t)\right|^\alpha\begin{cases}
			(2^k)^{1-s}, & \alpha > 0 ,\\(2^k)^{1-s+n\alpha}(2-k) & \alpha \leq 0,
		\end{cases}
		\lesssim t\left[\frac{\rho(x)}{t}\right]^s\left|B(x,t)\right|^\alpha.
	\end{align}
	On the other hand, Since $\rho(x)\backsim\rho(z)$, $\mu_{j,\Omega}(1)=0$, taking $\delta=\max\{1,\,\max\{0,n\alpha\}\}$, then Lemma \ref{lem:kjlguji}~(iii), Lemma \ref{lem:fcampanato1}~(i) and the H\"{o}lder inequality yield that
	\begin{align}\label{eq:9}
		{}&\Big|\int_{\left|x-z\right|\leq t}\left|\Omega(x-z)\right|K{_j^\mathcal{L}}(x,z)f_{B(x,t)}dz\Big|\\&
		\lesssim\left[\frac{t}{\rho(x)}\right]^{\delta}\Big(\int_{\left|x-z\right|\leq t}\left|\Omega(x-z)\right|^qdz\Big)^{\frac{1}{q}}\Big(\int_{\left|x-z\right|\leq t}\frac{1}{\left|x-z\right|^{(n-1)q'}}dy\Big)^{\frac{1}{q'}}\left|f_{B(x,t)}\right|\nonumber\\&
		\lesssim t\left[\frac{t}{\rho(x)}\right]^{\delta}\left|B(x,t)\right|^\alpha\begin{cases}\left(\frac{\rho\left(x\right)}{t}\right)^{n\alpha}, & \alpha>0, \\ 1+\log_2 \frac{\rho\left(x\right)}{t}, & \alpha \leq 0 ;
		\end{cases}
		\nonumber\\&
		\lesssim t\left[\frac{\rho(x)}{t}\right]^s\left|B(x,t)\right|^\alpha.
	\end{align}
	Summing up (\ref{eq:8}) and (\ref{eq:9}), we complete the proof for the case $t<\rho(x)$. The proof of Lemma \ref{lem:kernelguji2} is finished.
\end{proof}

\section{Proof of Theorem \ref{th1}}\label{section4}

The main ideas for the proof of Theorem \ref{th1} is inspired by Theorem 2 in \cite{GT1}, Theorem 4.1 in \cite{YYZ3} and Theorem 2.1-2.3 in \cite{HMY1}. 

\medskip

\textbf{Proof of Theorem \ref{th1}.}
Let $f\in\mathcal{E}{_\mathcal{L}^{\alpha,p}}(\mathbb{R}^n)$. Without loss of generality, assume that $\Vert f\Vert_{\mathcal{E}{_\mathcal{L}^{\alpha,p}}(\mathbb{R}^n)}=1.$
Let us start with the case when $\alpha\leq0$. For the balls $B=B(x_0,r)\in D_\rho,$ it is sufficient to show that
\begin{align}\label{eq:th1pf1}
	\int_{B}[\mu{_j^\mathcal{L}}(f)(x)]^pdx\lesssim\left|B\right|^{1+p\alpha}.
\end{align}
Decompose $f$ in the way that:
\begin{align}\label{eq:th1pfffenjie1}
	f=f\mathcal{X}_{B^*}+f\mathcal{X}_{(B^*)^c}=:f_1+f_2,
\end{align}
where $B^*$ is twice times extension of B with its center at $x_0$. It can be deduced from the $L^p(\mathbb{R}^n)\,(1<p<\infty)$ boundedness of $\mu{_j^\mathcal{L}}$ that
\begin{align}\label{eq:th1pf2}
	\int_{B}[\mu{_j^\mathcal{L}}(f_1)(x)]^pdx\lesssim\int_{B^*}\left|f(x)\right|^pdx\lesssim\left|B\right|^{1+p\alpha}.
\end{align}
Let $x\in B$, Lemma \ref{lem:fuzhuhanshu} indicates\ $\rho(x)\lesssim r$. By Lemma \ref{lem:kjlguji}~(i) and the H\"{o}lder inequality, take $l_1>\alpha n+1,$ then $\mu{_j^\mathcal{L}}(f_2)(x)$ is dominated by
\begin{align*}%\label{eq:th1pf3}
	{}& \bigg(\int_{r}^{\infty}\bigg|\int_{r\leq\left|x-y\right|\leq t}[\rho(x)]^{l_1}\frac{\left|f(y)\right|}{\left|x-y\right|^{n+l_1-1}}dy\bigg|^2\frac{dt}{t^3}\bigg)^\frac{1}{2}\nonumber\\
	&\lesssim[\rho(x)]^{l_1}r^{-l_1+1}\bigg(\int_{r}^{\infty}\bigg|\sum_{k=0}^{[\log_2\frac{t}{r}]}(2^{-l_1+1})^k\bigg(\frac{1}{\left|B(x,2^{k+1}r)\right|}\int_{B(x,2^{k+1}r)}\left|f(y)\right|^pdy\bigg)^{\frac{1}{p}}\bigg|^2\frac{dt}{t^3}\bigg)^\frac{1}{2}\nonumber\\
	&\leq[\rho(x)]^{l_1}r^{-l_1+1}\bigg(\int_{r}^{\infty}\bigg|\sum_{k=0}^{[\log_2\frac{t}{r}]}(2^{-l_1+1})^k\left|B(x,2^{k+1}r)\right|^\alpha\bigg|^2\frac{dt}{t^3}\bigg)^\frac{1}{2}\nonumber\\
	&\lesssim[\rho(x)]^{l_1}r^{-l_1+1}\left|B\right|^\alpha\Big(\int_{r}^{\infty}\bigg|\sum_{k=0}^{[\log_2\frac{t}{r}]}(2^{n\alpha-l_1+1})^k\bigg|^2\frac{dt}{t^3}\Big)^\frac{1}{2}\nonumber\\
	&\lesssim[\rho(x)]^{l_1}r^{-l_1+1}\left|B\right|^\alpha\left(\int_{r}^{\infty}\frac{dt}{t^3}\right)^\frac{1}{2}\lesssim\left|B\right|^\alpha.
\end{align*}
%\label{eq:th1pf4}
Therefore it follows that $$\int_{B}[\mu{_j^\mathcal{L}}(f_2)(x)]^pdx\lesssim\left|B\right|^{1+\alpha p}.$$ Combining this with (\ref{eq:th1pf2}) yields (\ref{eq:th1pf1}).

Consider the case $B=B(x_0,r)\notin D_\rho.$ We need to show 
\begin{align}\label{eq:th1pf5}
	\int_{B}\big[\mu{_j^\mathcal{L}}(f)(x)-\mathop{\rm{essinf}}\limits_{B}\mu{_j^\mathcal{L}}(f)\big]^pdx\lesssim\left|B\right|^{1+p\alpha}.
\end{align}
Since $$0\leq\mu{_j^\mathcal{L}}(f)-\mathop{\rm{essinf}}\limits_{B} \mu{_j^\mathcal{L}}(f)\leq\big\{[\mu{_j^\mathcal{L}}(f)]^2-\mathop{\rm{essinf}}\limits_{B} [\mu{_j^\mathcal{L}}(f)]^2\big\}^\frac{1}{2},$$ it suffices to show that $$\int_{B}\big\{[\mu{_j^\mathcal{L}}(f)(x)]^2-\mathop{\rm{essinf}}\limits_{B}[\mu{_j^\mathcal{L}}(f)]^2\big\}^\frac{p}{2}dx\lesssim\left|B\right|^{1+p\alpha}.$$ %\label{eq:th1pf6}
To this purpose, we write
\begin{align}\label{eq:th1pf7}
	[\mu{_j^\mathcal{L}}(f)(x)]^2&=\int_{0}^{8r}\bigg|\int_{\left|x-y\right|\leq t}K{_j^\mathcal{L}}(x,y)f(y)dy\bigg|^2\frac{dt}{t^3}+\int_{8r}^{\infty}\bigg|\int_{\left|x-y\right|\leq t}K{_j^\mathcal{L}}(x,y)f(y)dy\bigg|^2\frac{dt}{t^3}\nonumber\\
	&=:[\mu{_{j,r}^\mathcal{L}}(f)(x)]^2+[\mu{_{j,\infty}^\mathcal{L}}(f)(x)]^2.
\end{align}
Then, we obtain
\begin{align*}%\label{eq:th1pf8}
	&\int_{B}\Big\{[\mu{_j^\mathcal{L}}(f)(x)]^2-\mathop{\rm{essinf}}\limits_{B} [\mu{_j^\mathcal{L}}(f)]^2\Big\}^\frac{p}{2}dx\\&
	\lesssim\int_{B}[\mu{_{j,r}^\mathcal{L}}(f)(x)]^pdx+\int_{B}\Big\{[\mu{_{j,\infty}^\mathcal{L}}(f)(x)]^2-\mathop{\rm{essinf}}\limits_{B} [\mu{_{j,\infty}^\mathcal{L}}(f)]^2\Big\}^\frac{p}{2}dx\\&
	\lesssim\int_{B}[\mu{_{j,r}^\mathcal{L}}(f)(x)]^pdx+\left|B\right|\sup\limits_{x,y\in B}\left|[\mu{_{j,\infty}^\mathcal{L}}(f)(x)]^2-[\mu{_{j,\infty}^\mathcal{L}}(f)(y)]^2\right|^\frac{p}{2}.
\end{align*}
Hence, to show  (\ref{eq:th1pf5}), it is enough to prove
\begin{align}\label{eq:th1pf9}
	\int_{B}[\mu{_{j,r}^\mathcal{L}}(f)(x)]^pdx\lesssim\left|B\right|^{1+p\alpha},
\end{align}
and
\begin{align}\label{eq:th1pf10}
	\sup\limits_{x,y\in B}\left|[\mu{_{j,\infty}^\mathcal{L}}(f)(x)]^2-[\mu{_{j,\infty}^\mathcal{L}}(f)(y)]^2\right|^\frac{p}{2}\lesssim\left|B\right|^{p\alpha}.
\end{align}
Consider the inequality (\ref{eq:th1pf9}).  Split $f$ in the way
\begin{align}\label{eq:th1pfffenjie2}
	f=(f-f_B)\mathcal{X}_{B^*}+(f-f_B)\mathcal{X}_{(B^*)^c}+f_B=:f_1^*+f_2^*+f_B,
\end{align}
By the $L^p(\mathbb{R}^n)\,(1<p<\infty)$ boundedness of $\mu{_j^\mathcal{L}}$, (\ref{def:LMCS}) and Lemma \ref{lem:fcampanato1}~(ii), we have 
\begin{align}\label{eq:th1pf11}
		\int_{B}[\mu{_{j,r}^\mathcal{L}}(f_1^*)(x)]^pdx
	\lesssim\int_{B^*}\left|f(x)-f_{B^*}\right|^pdx+\int_{B^*}\left|f_{B^*}-f_B\right|^pdx\lesssim\left|B\right|^{1+p\alpha}.
\end{align}
Then Lemma \ref{lem:kjlguji}~(i), Lemma \ref{lem:fcampanato1}~(ii), (\ref{def:LMCS}) and the H\"{o}lder inequality yield that
\begin{align}\label{eq:th1pfn1}
		\Big|\int_{\left|x-y\right|\leq t}K{_j^\mathcal{L}}(x,y)f{_2^*}(y)dy\Big|=&\Big|\int_{r\leq\left|x-y\right|\leq t}K{_j^\mathcal{L}}(x,y)\left[f(y)-f_B\right]dy\Big|\nonumber\\
		\lesssim& \int_{r\leq\left|x-y\right|\leq t}\frac{\left|f(y)-f_B\right|}{\left|x-y\right|^{n-1}}dy\nonumber\\
		\lesssim&t\Big(\frac{t}{r}\Big)^{n-1}\Big\{\frac{1}{\left|B(x,t)\right|}\int_{B(x,t)}\left|f(y)-f_{B(x,t)}\right|dy+\frac{t}{r}\left|B\right|^\alpha\Big\}\\
		\lesssim&t\left(\frac{t}{r}\right)^{n-1}\Big\{\left|B(x,t)\right|^\alpha+\frac{t}{r}\left|B\right|^\alpha\Big\}\lesssim t\Big(\frac{t}{r}\Big)^{n}\left|B\right|^\alpha.\nonumber
\end{align}
Therefore, by (\ref{eq:th1pf7}) and the fact $r<t$, it hods that
\begin{align}\label{eq:th1pf12}
		\int_{B}[\mu{_{j,r}^\mathcal{L}}(f_2^*)(x)]^pdx\lesssim\int_{B}\Big[\int_{r}^{8r}\Big|t\left(\frac{t}{r}\right)^{n}\left|B\right|^\alpha\Big|^2\frac{dt}{t^3}\Big]^\frac{p}{2}dx\lesssim\left|B\right|^{1+p\alpha}.
	\end{align}
Using the Minkowski inequality and the fact that $\mu{_j}(1)=0$, we have 
\begin{align}\label{eq:fb1}
	\mu{_{j,r}^\mathcal{L}}(f_B)(x)\leq \bigg[\int_{0}^{8r}\bigg|\int_{\left|x-y\right|\leq t}\left[K{_j^\mathcal{L}}(x,y)-K_j(x,y)\right]f{_B}dy\bigg|^2\frac{dt}{t^3}\bigg]^\frac{1}{2}.
\end{align}
Lemma \ref{lem:kjlguji}~(iii) then implies $$\int_{\left|x-y\right|\leq t}|K{_j^\mathcal{L}}(x,y)-K_j(x,y)|dy\lesssim\int_{\left|x-y\right|\leq t}\frac{1}{|x-y|^{n-1}}\bigg(\frac{|x-y|}{\rho(y)}\bigg)^\delta dy.$$ Note that for any $x \in B,\ \rho(x)\backsim\rho(x_0)$. Since $\left|x-y\right|\leq t<8r<8\rho(x_0)$, then $\rho(x)\backsim\rho(y)$. Therefore, $\rho(x_0)\backsim\rho(x)\backsim\rho(y)$ and  $$
\int_{\left|x-y\right|\leq t}\frac{1}{|x-y|^{n-1}}\bigg(\frac{|x-y|}{\rho(y)}\bigg)^\delta dy\lesssim\bigg(\frac{t}{\rho(x_0)}\bigg)^\delta t.$$ On the other hand, by Lemma \ref{lem:fcampanato1}~(i): $\left|f_B\right|\lesssim\big(1+\log_2\frac{\rho(x_0)}{r}\big)\left|B\right|^\alpha$ and $r<\rho(x_0)$, we take $\delta=\delta_1>1$, then the left side in (\ref{eq:fb1}) is dominated by
$\left|B\right|^\alpha\big(\frac{r}{\rho(x_0)}\big)^{\delta_1-1}
	\leq\left|B\right|^\alpha.$
Hence, we have
\begin{align}\label{eq:th1pf13}
		\int_{B}[\mu{_{j,r}^\mathcal{L}}(f_B)(x)]^pdx\lesssim\left|B\right|^{1+p\alpha}.
	\end{align}
Suming up (\ref{eq:th1pf11}), (\ref{eq:th1pf12}) and (\ref{eq:th1pf13}), we obtain (\ref{eq:th1pf9}). 

It remains to prove inequality (\ref{eq:th1pf10}). For any $x,y\in B,$ it can be deduced from Lemma \ref{lem:kernelguji} whenever $s=0$ and $\alpha\leq0$ that 
\begin{align}\label{eq:th1pf14}
	&[\mu{_{j,\infty}^\mathcal{L}}(f)(x)]^2-[\mu{_{j,\infty}^\mathcal{L}}(f)(y)]^2\nonumber\\&
	\leq\int_{8r}^{\infty}\Big|\int_{\left|x-z\right|\leq t}K{_j^\mathcal{L}}(x,z)f(z)dz\Big|^2\frac{dt}{t^3}+\int_{8r}^{\infty}\Big|\int_{\left|y-z\right|\leq t}K{_j^\mathcal{L}}(y,z)f(z)dz\Big|^2\frac{dt}{t^3}\\&
	\lesssim\int_{8r}^{\infty}\Big|t\Big(\frac{t}{r}\Big)^{n\alpha}\left|B\right|^\alpha\Big|^2\frac{dt}{t^3}\lesssim\left|B\right|^{2\alpha}\nonumber,
\end{align}
which gives inequality  (\ref{eq:th1pf10}) and therefore, the proof of inequality (\ref{eq:th1pf5}) is finished.

Next, let us consider the case when $\alpha>0$. We need to prove (\ref{eq:th1pf1}), (\ref{eq:th1pf9}) and (\ref{eq:th1pf10}). The proof of (\ref{eq:th1pf1}) is totally the same as before. For (\ref{eq:th1pf9}), we only need to make some minor modifications to inequalities (\ref{eq:th1pf12}) and (\ref{eq:th1pf13}). By Lemma \ref{lem:fcampanato1}~(ii):\ $\left|f_{B(x,t)}-f_B\right|\lesssim\left(\frac{t}{r}\right)^{n\alpha}\left|B\right|^\alpha(\alpha>0)$, we can obtain the analogue of (\ref{eq:th1pfn1}) as follows:
\begin{align*}
	\Big|\int_{\left|x-y\right|\leq t}K{_j^\mathcal{L}}(x,y)f{_2^*}(y)dy\Big|
	\lesssim&t\Big(\frac{t}{r}\Big)^{n-1}\Big\{\frac{1}{\left|B(x,t)\right|}\int_{B(x,t)}\left|f(y)-f_{B(x,t)}\right|dy+\left(\frac{t}{r}\right)^{n\alpha}\left|B\right|^\alpha\Big\}\\
	\lesssim&t\Big(\frac{t}{r}\Big)^{n-1}\Big\{\left|B(x,t)\right|^\alpha+\left(\frac{t}{r}\right)^{n\alpha}\left|B\right|^\alpha\Big\}\lesssim t\Big(\frac{t}{r}\Big)^{n+n\alpha-1}\left|B\right|^\alpha.
\end{align*}
Similarly we have
\begin{align*}
	\int_{B}[\mu{_{j,r}^\mathcal{L}}(f_2^*)(x)]^pdx\lesssim\int_{B}\Big[\int_{r}^{8r}\Big|t\Big(\frac{t}{r}\Big)^{n+n\alpha-1}\left|B\right|^\alpha\Big|^2\frac{dt}{t^3}\Big]^\frac{p}{2}dx\lesssim\left|B\right|^{1+p\alpha}.
\end{align*}
Note that the difference in proving (\ref{eq:th1pf13}) lies in Lemma \ref{lem:fcampanato1}~(i):\ $\left|f_B\right|\lesssim\big(\frac{\rho(x_0)}{r}\big)^{n\alpha}\left|B\right|^\alpha$. Since $r<\rho(x_0)$, we take $\delta=\delta_2>n\alpha$ and now the expression in (\ref{eq:fb1}) is controled by $\big(\frac{\rho(x_0)}{r}\big)^{n\alpha}\left|B\right|^\alpha\big[\int_{0}^{4r}\big(\frac{t}{\rho(x_0)}\big)^{2\delta_2} t^2\frac{dt}{t^3}\big]^\frac{1}{2}
	\lesssim\left|B\right|^\alpha.$ Consequently, it remains to prove (\ref{eq:th1pf10}). Note that in the proof of the case $\alpha\leq0$,   (\ref{eq:th1pf10}) holds since Lemma \ref{lem:kernelguji}  is true. So in the case $\alpha>0$, we cannot use (\ref{eq:th1pf14}) to obtain (\ref{eq:th1pf10}) anymore and more delicate decomposition will be needed. 
	
	For any $x,y\in B=B(x_0,r)\notin D_\rho$. Lemma \ref{lem:fuzhuhanshu} gives $\rho(x)\backsim\rho(x_0)\backsim\rho(y)$ and
\begin{align*}
&	\left|[\mu{_{j,\infty}^\mathcal{L}}(f)(x)]^2-[\mu{_{j,\infty}^\mathcal{L}}(f)(y)]^2\right|\\&
	\lesssim\int_{8r}^{\infty}\Big|\int_{\left|x-z\right|\leq t}K{_j^\mathcal{L}}(x,z)f(z)dz-\int_{\left|y-z\right|\leq t}K{_j^\mathcal{L}}(y,z)f(z)dz\Big|\\
	&\times\bigg\{\Big|\int_{\left|x-z\right|\leq t}K{_j^\mathcal{L}}(x,z)f(z)dz\Big|+\Big|\int_{\left|y-z\right|\leq t}K{_j^\mathcal{L}}(y,z)f(z)dz\Big|\bigg\}\frac{dt}{t^3}.
\end{align*}
For any $ x\in B=B(x_0,r)$ and $t\geq8r$, denote 
\begin{align*}
	E_f(x,t):=\int_{\left|x-z\right|\leq t}K{_j^\mathcal{L}}(x,z)f(z)dz.
\end{align*}
Then Lemma \ref{lem:kernelguji} implies
that for any $x\in\mathbb{R}^n,\ t>0,\ 0\leq s<\min\{1,\,1+n\alpha\}$, there exists a positive constant such that for any $f\in\mathcal{E}{_\mathcal{L}^{\alpha,p}}(\mathbb{R}^n),$ 
\begin{align}\label{eq:th2pf10}
	\left|E_f(x,t)\right|\lesssim t\left[\frac{\rho(x)}{t}\right]^s\left|B(x,t)\right|^\alpha\lesssim t\left[\frac{\rho(x)}{t}\right]^s\left(\frac{t}{r}\right)^{n\alpha}\left|B\right|^\alpha.
\end{align}
For each fixed $x\in B$ and $ t\geq8r$, set
\begin{align*}
	H_f(x,y,t):=\left|E_f(x,t)-E_f(x,8r)-[E_f(y,t)-E_f(y,8r)]\right|.
\end{align*}
Since $$\left|E_f(x,t)-E_f(y,t)\right|\leq\left|E_f(x,8r)\right|+\left|E_f(y,8r)\right|+H_f(x,y,t),$$ the following task is to estimate the right side of the inequality respectively. Indeed, applying Lemma \ref{lem:kernelguji}, we have $\left|E_f(x,8r)\right|\lesssim r\left|B\right|^\alpha,\ \left|E_f(y,8r)\right|\lesssim r\left|B\right|^\alpha.$ Thus when $n\alpha-1<0$, combining (\ref{eq:th2pf10}) with $s=0$ yields that
\begin{align}\label{eq:th2pf5}
		\int_{8r}^{\infty}\left|E_f(x,8r)\right|\left|E_f(x,t)\right|\frac{dt}{t^3}\lesssim\int_{8r}^{\infty}r\left|B\right|^\alpha t\left(\frac{t}{r}\right)^{n\alpha}\left|B\right|^\alpha\frac{dt}{t^3}
		\lesssim\left|B\right|^{2\alpha}.
	\end{align}
Similarly, we obtain
\begin{align}\label{eq:th2pf6}
	\int_{8r}^{\infty}\left|E_f(y,8r)\right|\left|E_f(x,t)\right|\frac{dt}{t^3}\lesssim\left|B\right|^{2\alpha}.
\end{align}
Therefore the proof of inequality (\ref{eq:th1pf10}) can be reduced to prove that for any $x,y\in B$,
\begin{align}\label{eq:th2pf7}
	\int_{8r}^{\infty}H_f(x,y,t)\left|E_f(x,t)\right|\frac{dt}{t^3}\lesssim\left|B\right|^{2\alpha}.
\end{align}
To prove (\ref{eq:th2pf7}), we need to make a delicate decomposition on the integration region. Let 
\begin{align*}\quad \quad \quad  E_1=\{8r\leq\left|x-z\right|\leq t, \left|y-z\right|\geq t\}\ ,\quad E_2=
	\{8r\leq\left|y-z\right|\leq t, \left|x-z\right|\geq t\}, \quad\ \ \quad \quad \quad  \end{align*} 
\begin{align*}E_3=	\{8r\leq\left|x-z\right|\leq t, \left|y-z\right|\leq 8r\},\quad E_4=\{8r\leq\left|y-z\right|\leq t,\left|x-z\right|\leq 8r\},
	\ \end{align*}  $$ E_5=\{8r\leq\left|x-z\right|\leq t, 8r\leq\left|y-z\right|\leq t\},
\qquad\qquad\qquad\qquad\qquad\qquad\qquad\qquad\   $$
Now we dominate $	H_f(x,y,t)$ by five terms,
\begin{align*}
&\sum_{i=1}^4\Big|\int_{E_i}\left[K{_j^\mathcal{L}}(y,z)f(z)\right]dz\Big|+ \Big|\int_{E_5}\left[K{_j^\mathcal{L}}(x,z)-K{_j^\mathcal{L}}(y,z)\right]\left[f(z)\right]dz\Big|
	=:\sum_{i=1}^5H_{f,i}(x,y,t).
\end{align*}
Consider first the contribution of $H_{f,1}(x,y,t)$. because of $\mu_j(1)=0$ and (\ref{eq:th2pf10}), we get
\begin{align*}
	\int_{8r}^{\infty}H_{f,1}(x,y,t)\left|E_f(x,t)\right|\frac{dt}{t^3}&
	\lesssim\frac{\left|B\right|^\alpha}{r^{n\alpha}}\Big|\int_{8r}^{\infty}\int_{E_1}\left[K{_j^\mathcal{L}}(x,z)\left(f(z)-f_{B(x,r)}\right)\right]dz\Big|\frac{dt}{t^{2-n\alpha}}\\
	&\quad+\frac{\left|B\right|^\alpha}{r^{n\alpha}}\Big|\int_{8r}^{\infty}\int_{E_1}\left[K{_j^\mathcal{L}}(x,z)-K_{j}(x,z) \right]f_{B(x,r)}dz\Big|\frac{dt}{t^{2-n\alpha}}\\&
	=:I+II.
\end{align*}
First, let us consider $I$. By the Minkowski  inequality and Lemma \ref{lem:kjlguji}~(i), we obtain 
\begin{align*}
	I&\leq\frac{\left|B\right|^\alpha}{r^{n\alpha}}\int_{(B(x,4r))^c}\left|K{_j^\mathcal{L}}(x,z)\left(f(z)-f_{B(x,r)}\right)\right|\bigg(\int_{\left|x-z\right|}^{\left|y-z\right|}\frac{dt}{t^{2-n\alpha}}\bigg)dz\\
	&\lesssim\frac{\left|B\right|^\alpha}{r^{n\alpha}}\int_{(B(x,4r))^c}\frac{\left|f(z)-f_{B(x,r)}\right|}{\left|x-z\right|^{n-1}}\left|\frac{1}{\left|y-z\right|^{1-n\alpha}}-\frac{1}{\left|x-z\right|^{1-n\alpha}}\right|dz.
\end{align*}
Note that $z\in(B(x,4r))^c$, $x,y\in B(x_0,r)\subset B(x,2r),$ it follows that $\left|x-z\right|\backsim\left|y-z\right|$. The Mean Value Theorem then gives
\begin{align}\label{eq:th2pf11}
	\bigg|\frac{1}{\left|y-z\right|^{1-n\alpha}}-\frac{1}{\left|x-z\right|^{1-n\alpha}}\bigg|\lesssim\frac{r}{\left|x-z\right|^{2-n\alpha}}.
\end{align}
This, combining with (\ref{def:LMCS}) and Lemma \ref{lem:fcampanato1}~(ii), when $2n\alpha-1<0$, yields that
\begin{align*}
	I\lesssim& r\frac{\left|B\right|^\alpha}{r^{n\alpha}}\sum_{k=1}^{\infty}\int_{B(x,2^{k+1}r)\backslash B(x,2^kr)}\frac{\left|f(z)-f_{B(x,2^{k+1}r)}\right|+\left|f_{B(x,2^{k+1}r)}-f_{B(x,r)}\right|}{\left|x-z\right|^{n+1-n\alpha}}dz\\
	\lesssim&\left|B\right|^\alpha\sum_{k=1}^{\infty}\left(2^k\right)^{n\alpha-1}\left[\left|B(x,2^{k+1}r)\right|^\alpha+2^{kn\alpha}\left|B\right|^\alpha\right]
	\lesssim\left|B\right|^{2\alpha}.
\end{align*}
For the second term $II$. Using the Minkowski inequality, Lemma \ref{lem:fuzhuhanshu}, Lemma \ref{lem:kjlguji} (iii), Lemma \ref{lem:fcampanato1} and (\ref{eq:th2pf11}), and taking $n\alpha<\delta=\delta_4<\frac{1-n\alpha}{2}$ since $3n\alpha-1<0$, it follows that 
\begin{align*}
	II&\leq\frac{\left|B\right|^\alpha}{r^{n\alpha}}\int_{(B(x,4r))^c}\left|K{_j^\mathcal{L}}(x,z)-K_{j}(x,z)\right|\left| f_{B(x,r)}\right|\Big(\int_{\left|x-z\right|}^{\left|y-z\right|}\frac{dt}{t^{2-n\alpha}}\Big)dz\\
	&\lesssim\frac{\left|B\right|^{2\alpha}}{r^{n\alpha}}\Big(\frac{\rho(x_0)}{r}\Big)^{n\alpha}\int_{(B(x,4r))^c}\frac{1}{\left|x-z\right|^{n-\delta_4-1}\left(\rho(z)\right)^{\delta_4}}\bigg|\frac{1}{\left|y-z\right|^{1-n\alpha}}-\frac{1}{\left|x-z\right|^{1-n\alpha}}\bigg|dz\\
	&\lesssim\frac{\left|B\right|^{2\alpha}}{r^{n\alpha}}\Big(\frac{\rho(x_0)}{r}\Big)^{n\alpha}\int_{(B(x,4r))^c}\frac{r}{\left|x-z\right|^{n-\delta_4+1-n\alpha}\left(\rho(x)\right)^{\delta_4}}\Big(\frac{\rho(x)+\left|x-z\right|}{\rho(x)}\Big)^{l_0\delta_4}dz\\
	&\lesssim\frac{\left|B\right|^{2\alpha}}{r^{n\alpha}}\Big(\frac{\rho(x_0)}{r}\Big)^{n\alpha}\sum_{k=1}^{\infty}\int_{B(x,2^{k+1}r)\backslash B(x,2^kr)}\frac{r}{\left(2^kr\right)^{n-\delta_4+1-n\alpha}\left(\rho(x)\right)^{\delta_4}}\Big(1+\frac{2^{k+1}r}{\rho(x)}\Big)^{\delta_4}dz\\
	&\lesssim\left|B\right|^{2\alpha}\Big(\frac{\rho(x_0)}{r}\Big)^{n\alpha}\Big(\frac{r}{\rho(x_0)}\Big)^{\delta_4}\sum_{k=1}^{\infty}(2^k)^{n\alpha+\delta_4-1}2^{k\delta_4}\lesssim\left|B\right|^{2\alpha}.
\end{align*}
Therefore
\begin{align}\label{eq:th2pf12}
	\int_{8r}^{\infty}H_{f,1}(x,y,t)\left|E_f(x,t)\right|\frac{dt}{t^3}\lesssim\left|B\right|^{2\alpha}.
\end{align}
Similarly, we have
\begin{align}\label{eq:th2pf13}
	\int_{8r}^{\infty}H_{f,2}(x,y,t)\left|E_f(x,t)\right|\frac{dt}{t^3}\lesssim\left|B\right|^{2\alpha}.
\end{align}

Consider now the term $H_{f,3}(x,y,t)$.
By $\mu_j(1)=0$, we get 
\begin{align*}
	H_{f,3}(x,y,t)\leq&\Big|\int_{E_3}K{_j^\mathcal{L}}(x,z)\left[f(z)-f_{B(x,r)}\right]dz\Big|+\Big|\int_{E_3}\left[K{_j^\mathcal{L}}(x,z)-K_j(x,z)\right]f_{B(x,r)}dz\Big|.
\end{align*}
Notice that $x,y\in B(x_0,r),$ $\left|x-y\right|<2r$, then $\left|x-z\right|\leq\left|x-y\right|+\left|y-z\right|<10r<10\rho(x_0)$, which implies that $\rho(x_0)\backsim\rho(z).$ Using Lemma \ref{lem:kjlguji}~(i), (\ref{def:LMCS}) and Lemma \ref{lem:fcampanato1},  take $\delta=\delta_3>n\alpha$, we get
\begin{align*}
	H_{f,2}(x,y,t)
	\lesssim&\int_{8r\leq\left|x-z\right|\leq 10r}\frac{\left|f(z)-f_{B(x,r)}\right|}{\left|x-z\right|^{n-1}}dz+\int_{8r\leq\left|x-z\right|\leq 10r}\frac{\left|f_{B(x,r)}\right|}{\left|x-z\right|^{n-\delta_3-1}\left[\rho(z)\right]^{\delta_3}}dz\\
	\lesssim& \frac{1}{r^{n-1}}\int_{\left|x-z\right|\leq10r}\left|f(z)-f_{B(x,10r)}\right|dz+\frac{1}{r^{n-1}}\int_{\left|x-z\right|\leq10r}\left|f_{B(x,10r)}-f_{B(x,r)}\right|dz\\
	&+r\Big(\frac{r}{\rho(x_0)}\Big)^{\delta_3}\frac{1}{r^n}\int_{\left|x-z\right|\leq 10r}\Big(\frac{\rho(x_0)}{r}\Big)^{n\alpha}\left|B\right|^\alpha dz\\&\lesssim r\left|B\right|^\alpha.
\end{align*}
When $n\alpha-1<0$, this inequality, together with  (\ref{eq:th2pf10}) ($s=0$) yields
\begin{align}\label{eq:th2pf8}
		\int_{8r}^{\infty}H_{f,3}(x,y,t)\left|E_f(x,t)\right|\frac{dt}{t^3}\lesssim\int_{8r}^{\infty}r\left|B\right|^\alpha t\left(\frac{t}{r}\right)^{n\alpha}\left|B\right|^\alpha\frac{dt}{t^3}\lesssim\left|B\right|^{2\alpha}.
	\end{align}
The same reasoning applies to $H_{f,4}(x,y,t)$ gives
\begin{align}\label{eq:th2pf9}
	\int_{8r}^{\infty}H_{f,4}(x,y,t)\left|E_f(x,t)\right|\frac{dt}{t^3}\lesssim\left|B\right|^{2\alpha}.
\end{align}

It remains to consider the contribution of $H_{f,5}(x,y,t)$. From the fact $\mu_j(1)=0$ and Lemma \ref{lem:kernelguji}, taking $s=m$, where the range of $m$ will be given later, we obtain
\begin{align*}
	\int_{8r}^{\infty}H_{f,5}(x,y,t)\left|E_f(x,t)\right|\frac{dt}{t^3}&
	\lesssim\frac{\left|B\right|^\alpha}{r^{n\alpha}}\Big|\int_{8r}^{\infty}\int_{E_5}(K{_j^\mathcal{L}}(x,z)-K{_j^\mathcal{L}}(y,z))(f(z)-f_{B(x,r)})dz\Big|\frac{dt}{t^{2-n\alpha}}\\
	&\quad+\frac{\left[\rho(x)\right]^m\left|B\right|^\alpha}{r^{n\alpha}}\Big|\int_{8r}^{\infty}\int_{E_5}
	K_{j}(x,y,z)
	 f_{B(x,r)}dz\Big|\frac{dt}{t^{2-n\alpha+m}}\\&
	=:III+IV,
\end{align*}
where $K_{j}(x,y,z)=(K{_j^\mathcal{L}}(x,z)-K_{j}(x,z))-(K{_j^\mathcal{L}}(y,z)-K_{j}(y,z) )$.

For the term $III$. By the Minkowski inequality, the H\"{o}lder inequality, Lemma \ref{lem:kjlguji}~(ii) and Lemma \ref{lem:fcampanato1}, we take $\beta>2n\alpha$ and obtain 
\begin{align*}
	III\lesssim& r^\beta\frac{\left|B\right|^\alpha}{r^{n\alpha}}\int_{(B(x,4r))^c}\frac{\left|f(z)-f_{B(x,r)}\right|}{\left|x-z\right|^{n-1+\beta}}\Big(\int_{\left|x-z\right|}^{\infty}\frac{dt}{t^{2-n\alpha}}\Big)dz\\
	\lesssim& \left|B\right|^\alpha\sum_{k=1}^{\infty}\bigg\{\frac{\left(2^k\right)^{n\alpha-\beta}}{\left|B(x,2^{k+1}r)\right|}\int_{B(x,2^{k+1}r)}\left|f(z)-f_{B(x,2^{k+1}r)}\right|dz\\
	&+\frac{\left(2^k\right)^{n\alpha-\beta}}{\left|B(x,2^{k+1}r)\right|}\int_{B(x,2^{k+1}r)}\left|f_{B(x,2^{k+1}r)}-f_{B(x,r)}\right|dz\bigg\}\\
	\lesssim&\left|B\right|^\alpha\sum_{k=1}^{\infty}\left(2^k\right)^{n\alpha-\beta}\left[\left|B(x,2^{k+1}r)\right|^\alpha+2^{kn\alpha}\left|B\right|^\alpha\right]  \lesssim\left|B\right|^{2\alpha}.
\end{align*}
Let us consider $IV$. Since $x,y\in B(x_0,r),~8r\leq\left|x-z\right|\leq t, ~8r\leq\left|y-z\right|\leq t$, we obtain
$\left|x-z\right|\leq \frac{5}{4}\left|y-z\right|, \quad \left|y-z\right|\leq\frac{5}{4}\left|x-z\right|.
$ Hence $\left|x-z\right|\backsim\left|y-z\right|.$ By Remark \ref{lem:kjlcha2} and Lemma \ref{lem:kjlguji}~(iii), we have
\begin{align}\label{eq:keyxfj}
	|K_{j}(x,y,z)|&\lesssim\big|[\widetilde{K_j^\mathcal{L}}(x, z)-\widetilde{K_j^\Delta}(x, z)]-[\widetilde{K_j^\mathcal{L}}(y, z)-\widetilde{K_j^\Delta}(y, z)]\big||x-z|\nonumber\\&\quad+|\widetilde{K_j^\mathcal{L}}(y, z)-\widetilde{K_j^\Delta}(y, z)|\big||x-z|-|y-z|\big|\nonumber\\&
	\lesssim\frac{|x-y|^\gamma}{|x-z|^{n+\gamma-1}}\left(\frac{|x-z|}{\rho(x)}\right)^{\eta}+\frac{r}{|x-z|^{n}}\left(\frac{|x-z|}{\rho(z)}\right)^\delta.
\end{align}
Thus, $IV$ can be divided into two parts.
\begin{align*}
	IV\lesssim&\frac{\left[\rho(x)\right]^m\left|B\right|^\alpha}{r^{n\alpha}}\int_{8r}^{\infty}\Big|\int_{E_5}\frac{|x-y|^\gamma}{|x-z|^{n+\gamma-1}}\left(\frac{|x-z|}{\rho(x)}\right)^{\eta} f_{B(x,r)}dz\Big|\frac{dt}{t^{2-n\alpha+m}}\\
	&+\frac{\left[\rho(x)\right]^m\left|B\right|^\alpha}{r^{n\alpha}}\int_{8r}^{\infty}\Big|\int_{E_5}\frac{r}{|x-z|^{n}}\left(\frac{|x-z|}{\rho(z)}\right)^\delta f_{B(x,r)}dz\Big|\frac{dt}{t^{2-n\alpha+m}}\\
	=:&IV_1+IV_2.
\end{align*}
To treat $IV_1$, we choose $s=m$ in Lemma \ref{lem:kernelguji} with $0<m=\eta-\frac{\gamma}{2}<1$. It can be deduced from $2n\alpha-\gamma<0$ that $n\alpha-\eta+m<0$ and $n\alpha+\eta-\gamma-m<0$. Then by the Minkowski inequality and Remark \ref{lem:kjlcha2}, we have
\begin{align*}
	IV_1
	\lesssim&\left[\rho(x)\right]^m\frac{\left|B\right|^{2\alpha}}{r^{n\alpha}}\Big(\frac{\rho(x_0)}{r}\Big)^{n\alpha}\int_{(B(x,4r))^c}\frac{|x-y|^\gamma}{|x-z|^{n+\gamma-1}}\Big(\frac{|x-z|}{\rho(x)}\Big)^{\eta}\Big(\int_{\left|x-z\right|}^{\infty}\frac{dt}{t^{2-n\alpha+m}}\Big)dz\\
	\lesssim&\left[\rho(x)\right]^m\frac{\left|B\right|^{2\alpha}}{r^{n\alpha}}\Big(\frac{\rho(x_0)}{r}\Big)^{n\alpha}\sum_{k=1}^{\infty}\int_{B(x,2^{k+1}r)\backslash B(x,2^kr)}\frac{r^{\gamma}}{\left(2^kr\right)^{n+\gamma-\eta+m-n\alpha}\left(\rho(x)\right)^\eta}dz\\
	\lesssim&\left|B\right|^{2\alpha}\Big(\frac{r}{\rho(x_0)}\Big)^{\eta-n\alpha-m}\sum_{k=1}^{\infty}(2^k)^{n\alpha+\eta-\gamma-m}\\
	\lesssim&\left|B\right|^{2\alpha}\Big(\frac{r}{\rho(x_0)}\Big)^{\eta-n\alpha-m}\lesssim\left|B\right|^{2\alpha}.
\end{align*}
For $IV_2$, we pick $m=0$. By Minkowski integral inequality, we have 
\begin{align*}
	IV_2&\lesssim\frac{\left|B\right|^{2\alpha}}{r^{n\alpha}}\Big(\frac{\rho(x_0)}{r}\Big)^{n\alpha}\int_{(B(x,4r))^c}\frac{r}{\left|x-z\right|^{n-\delta+1-n\alpha}\left(\rho(z)\right)^{\delta}}dz
\end{align*}
Using the same method as what we have done for $II$ in the case $m=0$, it follows that $IV_2\lesssim\left|B\right|^{2\alpha}$ and therefore
\begin{align}\label{eq:th2pf14}
	\int_{8r}^{\infty}H_{f,5}(x,y,t)\left|E_f(x,t)\right|\frac{dt}{t^3}\lesssim\left|B\right|^{2\alpha}.
\end{align}
Then (\ref{eq:th2pf7}) follows from (\ref{eq:th2pf8}), (\ref{eq:th2pf9}), (\ref{eq:th2pf12}), (\ref{eq:th2pf13}) and (\ref{eq:th2pf14}).
Combining (\ref{eq:th2pf7})  with (\ref{eq:th2pf5}) and (\ref{eq:th2pf6}), we obtain (\ref{eq:th1pf10}) and  complete the proof of Theorem \ref{th1}.
\qed

\section{Proof of Theorem \ref{th3} }\label{section5}

We are in the position to prove Theorem \ref{th3}. When $\alpha\leq 0$, we do not need to assume the smoothness condition to obtain the boundedness of $\mu{_{j,\Omega}^\mathcal{L}}$. However, for the case $\alpha>0$, inspired by \cite{DLX2}, we need to utilize a kind of $L^q$-Dini condition (\ref{def:lqDini2}). 

\medskip

\textbf{Proof of Theorem \ref{th3}.}
Assume without loss of generality that $f\in\mathcal{E}{_\mathcal{L}^{\alpha,p}}(\mathbb{R}^n)$ and $\Vert f\Vert_{\mathcal{E}{_\mathcal{L}^{\alpha,p}}(\mathbb{R}^n)}=1.$ We begin with the case $\alpha\leq0$, which will be proved in a way similar to Theorem \ref{th1}.  For the ball $B=B(x_0,r)\in D_\rho,$ it suffices to show that
\begin{align}\label{eq:th3pf1}
	\int_{B}[\mu{_{j,\Omega}^\mathcal{L}}(f)(x)]^pdx\lesssim\left|B\right|^{1+p\alpha}.
\end{align}
Decompose $f$ in the same way as in (\ref{eq:th1pfffenjie1}). By the $L^p(\mathbb{R}^n)$ boundedness of $\mu{_{j,\Omega}^\mathcal{L}}$, we have 
\begin{align}\label{eq:th3pf2}
	\int_{B}[\mu{_{j,\Omega}^\mathcal{L}}(f_1)(x)]^pdx\lesssim\int_{B^*}\left|f(x)\right|^pdx\lesssim\left|B\right|^{1+p\alpha}.
\end{align}
Let $x\in B$, then\ $\rho(x)\lesssim r$. By Lemma \ref{lem:kjlguji}~(i) and the H\"{o}lder inequality, we choose $l_2>\alpha n+1,$ and can dominate $\mu{_{j,\Omega}^\mathcal{L}}(f_2)(x)$ by
\begin{align*}
& \bigg(\int_{r}^{\infty}\bigg|\int_{r\leq\left|x-y\right|\leq t}\big|\Omega(x-y)\big|[\rho(x)]^{l_2}\frac{|f(y)|}{\left|x-y\right|^{n+l_2-1}}dy\bigg|^2\frac{dt}{t^3}\bigg)^\frac{1}{2}\nonumber\\&
	\lesssim [\rho(x)]^{l_2}r^{-l_2+1}\bigg[\int_{r}^{\infty}\bigg|\sum_{k=0}^{[\log_2\frac{t}{r}]}(2^{-l_2+1})^k\Big(\frac{1}{\left|B(x,2^{k+1}r)\right|}\int_{B(x,2^{k+1}r)}\left|\Omega(x-y)\right|^{q}dy\Big)^{\frac{1}{q}}\nonumber\\
	&\times\Big(\frac{1}{\left|B(x,2^{k+1}r)\right|}\int_{B(x,2^{k+1}r)}\left|f(y)\right|^{q'}dy\Big)^{\frac{1}{q'}}\bigg|^2\frac{dt}{t^3}\bigg]^\frac{1}{2}\nonumber\\&
	\lesssim[\rho(x)]^{l_2}r^{-l_2+1}\bigg[\int_{r}^{\infty}\bigg|\sum_{k=0}^{[\log_2\frac{t}{r}]}(2^{-l_2+1})^k\left(\frac{1}{\left|B(x,2^{k+1}r)\right|}\int_{B(x,2^{k+1}r)}\left|f(y)\right|^{p}dy\right)^{\frac{1}{p}}\bigg|^2\frac{dt}{t^3}\bigg]^\frac{1}{2}\nonumber\\&
	\lesssim[\rho(x)]^{l_2}r^{-l_2+1}\left|B\right|^\alpha\bigg(\int_{r}^{\infty}\bigg|\sum_{k=0}^{[\log_2\frac{t}{r}]}(2^{n\alpha-l_2+1})^k\bigg|^2\frac{dt}{t^3}\bigg)^\frac{1}{2}\lesssim \left|B\right|^\alpha.
\end{align*}
Therefore
\begin{align}\label{eq:th3pf4}
	\int_{B}[\mu{_{j,\Omega}^\mathcal{L}}(f_2)(x)]^pdx\lesssim\left|B\right|^{1+\alpha p}.
\end{align}
Combining this with (\ref{eq:th3pf2}) yields (\ref{eq:th3pf1}).

Now assume $B=B(x_0,r)\notin D_\rho,$ we claim that 
\begin{align}\label{eq:th3pf5}
	\int_{B}\big[\mu{_{j,\Omega}^\mathcal{L}}(f)(x)-\mathop{\rm{essinf}}\limits_{B} \mu{_{j,\Omega}^\mathcal{L}}(f)\big]^pdx\lesssim\left|B\right|^{1+p\alpha}.
\end{align}
Since $0\leq\mu{_{j,\Omega}^\mathcal{L}}(f)-\mathop{\rm{essinf}}\limits_{B} \mu{_{j,\Omega}^\mathcal{L}}(f)\leq\big\{[\mu{_{j,\Omega}^\mathcal{L}}(f)]^2-\mathop{\rm{essinf}}\limits_{B} [\mu{_{j,\Omega}^\mathcal{L}}(f)]^2\big\}^\frac{1}{2},$ it suffices to show that $$\int_{B}\big\{[\mu{_{j,\Omega}^\mathcal{L}}(f)(x)]^2-\mathop{\rm{essinf}}\limits_{B} [\mu{_{j,\Omega}^\mathcal{L}}(f)]^2\big\}^\frac{p}{2}dx\lesssim\left|B\right|^{1+p\alpha}.$$ To this aim, write %\label{eq:th3pf6}
\begin{align}\label{eq:th3pf7}
[\mu{_{j,\Omega}^\mathcal{L}}(f)(x)]^2\nonumber&
	=\int_{0}^{8r}\bigg|\int_{\left|x-y\right|\leq t}\left|\Omega(x-y)\right|K{_j^\mathcal{L}}(x,y)f(y)dy\bigg|^2\frac{dt}{t^3}\\& \quad+\int_{8r}^{\infty}\bigg|\int_{\left|x-y\right|\leq t}\left|\Omega(x-y)\right|K{_j^\mathcal{L}}(x,y)f(y)dy\bigg|^2\frac{dt}{t^3}\nonumber\\&
	=:[\mu{_{j,\Omega;r}^\mathcal{L}}(f)(x)]^2+[\mu{_{j,\Omega;\infty}^\mathcal{L}}(f)(x)]^2.
\end{align}
Then, similar to Theorem \ref{th1},  to prove (\ref{eq:th3pf5}), it is enough to show 
\begin{align}\label{eq:th3pf8}
	\int_{B}[\mu{_{j,\Omega;r}^\mathcal{L}}(f)(x)]^pdx\lesssim\left|B\right|^{1+p\alpha},
\end{align}
and
\begin{align}\label{eq:th3pf9}
	\sup\limits_{x,y\in B}\left|[\mu{_{j,\Omega;\infty}^\mathcal{L}}(f)(x)]^2-[\mu{_{j,\Omega;\infty}^\mathcal{L}}(f)(y)]^2\right|^\frac{p}{2}\lesssim\left|B\right|^{p\alpha}.
\end{align}
Consider to estimate (\ref{eq:th3pf8}). Write $f$ in the same form as in (\ref{eq:th1pfffenjie2}). It follows from the $L^p(\mathbb{R}^n)\,(1<p<\infty)$ boundedness of $\mu{_{j,\Omega}^\mathcal{L}}$, (\ref{def:LMCS}) and Lemma \ref{lem:fcampanato1}~(ii) that
\begin{align}\label{eq:th3pf10}
		\int_{B}[\mu{_{j,\Omega;r}^\mathcal{L}}(f_1^*)(x)]^pdx&\leq\int_{B}[\mu{_{j,\Omega}^\mathcal{L}}(f_1^*)(x)]^pdx\lesssim\int_{B^*}\left|f(x)-f_B\right|^pdx\lesssim\left|B\right|^{1+p\alpha}.
	\end{align}
Lemma \ref{lem:kjlguji}~(i), Lemma \ref{lem:fcampanato1}~(ii), (\ref{def:LMCS}) and the H\"{o}lder inequality then give
\begin{align}\label{eq:th3pfn1}
	&\bigg|\int_{\left|x-y\right|\leq t}\left|\Omega(x-y)\right|K{_j^\mathcal{L}}(x,y)f{_2^*}(y)dy\bigg|\nonumber\\&
	\lesssim\frac{1}{r^{n-1}}\bigg(\int_{\left|x-y\right|\leq t}\left|\Omega(x-y)\right|\left|f(y)-f_{B(x,t)}\right|dy+\int_{\left|x-y\right|\leq t}\left|\Omega(x-y)\right|\left|f_{B(x,t)}-f_{B}\right|dy\bigg)\nonumber\\&
	\lesssim t\left(\frac{t}{r}\right)^{n-1}\bigg[\bigg(\frac{1}{\left|B(x,t)\right|}\int_{B(x,t)}\left|f(y)-f_{B(x,t)}\right|^{p}dy\bigg)^{\frac{1}{p}}+\frac{t}{r}\left|B\right|^\alpha\bigg]
	\lesssim t\left(\frac{t}{r}\right)^{n}\left|B\right|^\alpha.
\end{align}
Combining this with (\ref{eq:th3pf7}) and $r<t$, it follows that
\begin{align}\label{eq:th3pf11}
		\int_{B}[\mu{_{j,\Omega;r}^\mathcal{L}}(f_2^*)(x)]^pdx\lesssim\int_{B}\bigg(\int_{r}^{8r}\left|t\left(\frac{t}{r}\right)^{n}\left|B\right|^\alpha\right|^2\frac{dt}{t^3}\bigg)^\frac{p}{2}dx\lesssim\left|B\right|^{1+p\alpha}.
	\end{align}
By the Minkowski inequality and $\mu{_{j,\Omega}}(1)=0$, we have 
\begin{align*}
	\mu{_{j,\Omega;r}^\mathcal{L}}(f_B)(x)\leq& \bigg(\int_{0}^{8r}\bigg|\int_{\left|x-y\right|\leq t}\big|\Omega(x-y)\big|\left[K{_j^\mathcal{L}}(x,y)-K_j(x,y)\right]f{_B}dy\bigg|^2\frac{dt}{t^3}\bigg)^\frac{1}{2}.
\end{align*}
Similar to Theorem \ref{th1}, using the H\"{o}lder inequality, we can also obtain the inner layer integral part is bounded by $ \big(\frac{t}{\rho(x_0)}\big)^\delta t\left|f_B\right|.$ Taking $\delta=\delta_5>1$, it follows from Lemma \ref{lem:fcampanato1}~(i) and $r<\rho(x_0)$ that $\mu{_{j,\Omega;r}^\mathcal{L}}(f_B)(x)\lesssim\left|B\right|^\alpha\left(\frac{r}{\rho(x_0)}\right)^{\delta_5-1}
\leq\left|B\right|^\alpha.$
Thus,
\begin{align}\label{eq:th3pf12}
	\int_{B}[\mu{_{j,\Omega;r}^\mathcal{L}}(f_B)(x)]^pdx\lesssim\left|B\right|^{1+p\alpha}.
\end{align}
Suming up (\ref{eq:th3pf10}), (\ref{eq:th3pf11}), (\ref{eq:th3pf12}), we obtain (\ref{eq:th3pf8}). In addition, (\ref{eq:th3pf9}) follows from Lemma \ref{lem:kernelguji2} for $s=0$ and $\alpha\leq0$. Therefore, (\ref{eq:th3pf5}) is proved.

Next, let us consider the case $\alpha>0$. We need to prove (\ref{eq:th3pf1}), (\ref{eq:th3pf8}) and (\ref{eq:th3pf9}). (\ref{eq:th3pf1}) is totally the same as before. For (\ref{eq:th3pf8}), it is similar to Theorem \ref{th1} so we only need to make some minor modifications to get (\ref{eq:th3pf11}) and (\ref{eq:th3pf12}). By Lemma \ref{lem:fcampanato1}~(ii):\ $\left|f_{B(x,t)}-f_B\right|\lesssim\left(\frac{t}{r}\right)^{n\alpha}\left|B\right|^\alpha(\alpha>0)$, we can obtain the analogue of (\ref{eq:th3pfn1}) as follow:
\begin{align*}
	&\Big|\int_{\left|x-y\right|\leq t}\left|\Omega(x-y)\right|K{_j^\mathcal{L}}(x,y)f{_2^*}(y)dy\Big|\\&
	\lesssim t\left(\frac{t}{r}\right)^{n-1}\Bigg\{\Big(\frac{1}{\left|B(x,t)\right|}\int_{B(x,t)}\left|f(y)-f_{B(x,t)}\right|^{p}dy\Big)^{\frac{1}{p}}+\left(\frac{t}{r}\right)^{n\alpha}\left|B\right|^\alpha\Bigg\}\\&
	\lesssim t\left(\frac{t}{r}\right)^{n+n\alpha-1}\left|B\right|^\alpha.
\end{align*}
Similarly, we have $$\int_{B}[\mu{_{j,\Omega;r}^\mathcal{L}}(f_2^*)(x)]^pdx\lesssim\int_{B}\Big[\int_{r}^{8r}\big|t\left(\frac{t}{r}\right)^{n+n\alpha-1}\left|B\right|^\alpha\big|^2\frac{dt}{t^3}\Big]^\frac{p}{2}dx\nonumber
\lesssim\left|B\right|^{1+p\alpha}.$$%\label{eq:th4pf2}
The difference in proving (\ref{eq:th3pf12}) is due to lemma \ref{lem:fcampanato1}~(i). Therefore, for $r<\rho(x_0)$, we take $\delta=\delta_6>n\alpha$ and it follows $\mu{_{j,\Omega;r}^\mathcal{L}}(f_B)(x)\lesssim\left|B\right|^\alpha\big(\frac{r}{\rho(x_0)}\big)^{\delta_6-n\alpha}
\leq\left|B\right|^\alpha.$

It remains to prove (\ref{eq:th3pf9}). Since $\alpha>0$, we need to use more complicated decomposition as in the proof of Theorem \ref{th1}. For any $x,y\in B=B(x_0,r)\notin D_\rho,$ it follows from Lemma \ref{lem:fuzhuhanshu} that $\rho(x)\backsim\rho(x_0)\backsim\rho(y).$ Then\begin{align*}
&\left|[\mu{_{j,\Omega;\infty}^\mathcal{L}}(f)(x)]^2-[\mu{_{j,\Omega;\infty}^\mathcal{L}}(f)(y)]^2\right|\\&
	\leq\int_{8r}^{\infty}\Big|\int_{\left|x-z\right|\leq t}\left|\Omega(x-z)\right|K{_j^\mathcal{L}}(x,z)f(z)dz-\int_{\left|y-z\right|\leq t}\left|\Omega(y-z)\right|K{_j^\mathcal{L}}(y,z)f(z)dz\Big|\\
	&\times\bigg(\Big|\int_{\left|x-z\right|\leq t}\left|\Omega(x-z)\right|K{_j^\mathcal{L}}(x,z)f(z)dz\Big|+\Big|\int_{\left|y-z\right|\leq t}\left|\Omega(y-z)\right|K{_j^\mathcal{L}}(y,z)f(z)dz\Big|\bigg)\frac{dt}{t^3}.
\end{align*}
For any $ x\in B=B(x_0,r)$ and $t\geq8r$, use the same notation as in Theorem \ref{th1}, 
\begin{align*}
	E_f(x,t):=\int_{\left|x-z\right|\leq t}\left|\Omega(x-z)\right|K{_j^\mathcal{L}}(x,z)f(z)dz.
\end{align*}
It follows from Lemma \ref{lem:kernelguji2} that for any $x\in\mathbb{R}^n,\ t>0,\ 0\leq s<\min\{1,\,1+n\alpha\}$, there exists a positive constant such that for any $f\in\mathcal{E}{_\mathcal{L}^{\alpha,p}}(\mathbb{R}^n),$ we have 
\begin{align}\label{eq:th4pf5}
	\left|E_f(x,t)\right|\lesssim t\left[\frac{\rho(x)}{t}\right]^s\left|B(x,t)\right|^\alpha\lesssim t\left[\frac{\rho(x)}{t}\right]^s\left(\frac{t}{r}\right)^{n\alpha}\left|B\right|^\alpha.
\end{align}
For each fixed $x\in B$ and $ t\geq8r$, we set
\begin{align*}
	H_f(x,y,t):=\left|E_f(x,t)-E_f(x,8r)-[E_f(y,t)-E_f(y,8r)]\right|.
\end{align*}
Likewise, $\left|E_f(x,t)-E_f(y,t)\right|\leq\left|E_f(x,8r)\right|+\left|E_f(y,8r)\right|+H_f(x,y,t),$ and we will estimate the right side of the inequality respectively. By (\ref{eq:th4pf5}), when $n\alpha-1<0$, we get
\begin{align}\label{eq:th4pf6}
	\int_{8r}^{\infty}\left|E_f(x,8r)\right|\left|E_f(x,t)\right|\frac{dt}{t^3}\lesssim\int_{8r}^{\infty}r\left|B\right|^\alpha t\left(\frac{t}{r}\right)^{n\alpha}\left|B\right|^\alpha\frac{dt}{t^3}\lesssim\left|B\right|^{2\alpha},
\end{align}
and
\begin{align}\label{eq:th4pf7}
	\int_{8r}^{\infty}\left|E_f(y,8r)\right|\left|E_f(x,t)\right|\frac{dt}{t^3}\lesssim\left|B\right|^{2\alpha}.
\end{align}
To obtain (\ref{eq:th3pf9}), it suffices to prove that for any $x,y\in B$,
\begin{align}\label{eq:th4pf8}
	\int_{8r}^{\infty}H_f(x,y,t)\left|E_f(x,t)\right|\frac{dt}{t^3}\lesssim\left|B\right|^{2\alpha}.
\end{align}
Indeed, decompose $H_f(x,y,t)$ in the following way
\begin{align*}
	H_f(x,y,t)
	\lesssim
	&\sum_{i=1}^4
	\Big|\int_{E_i}\left|\Omega(x-z)\right|\left[K{_j^\mathcal{L}}(x,z)f(z)\right]dz\Big|
\\&+\Big|\int_{E_5}\left[\left|\Omega(x-z)\right|K{_j^\mathcal{L}}(x,z)-\left|\Omega(y-z)\right|K{_j^\mathcal{L}}(y,z)\right]\left[f(z)\right]dz\Big|\\&	=:\sum_{i=1}^5 H_{f,i}(x,y,t).
\end{align*}
We need to consider the contributions of each term of the summation. 

\noindent {\bf Contributions of $H_{f,1}(x,y,t)$ and $H_{f,2}(x,y,t)$}. The fact $\mu_{j,\Omega}(1)=0$ and (\ref{eq:th4pf5}) give
\begin{align*}
	&\int_{8r}^{\infty}H_{f,1}(x,y,t)\left|E_f(x,t)\right|\frac{dt}{t^3}\\&
	\lesssim\frac{\left|B\right|^\alpha}{r^{n\alpha}}\int_{8r}^{\infty}\Big|\int_{E_1}\left|\Omega(x-z)\right|\left[K{_j^\mathcal{L}}(x,z)\left(f(z)-f_{B(x,r)}\right)\right]dz\Big|\frac{dt}{t^{2-n\alpha}}\\
	&+\frac{\left|B\right|^\alpha}{r^{n\alpha}}\int_{E_1}\left|\Omega(x-z)\right|\left[K{_j^\mathcal{L}}(x,z)-K_{j}(x,z) \right]f_{B(x,r)}dz\Big|\frac{dt}{t^{2-n\alpha}}\\&
	=:I_1+I_2.
\end{align*}
For the first term $I_1$. When $2n\alpha-1<0$, by the Minkowski inequality, Lemma \ref{lem:kjlguji}~(i), (\ref{eq:th2pf11}) and Lemma \ref{lem:fcampanato1} (ii), we obtain 
\begin{align*}
	I_1\lesssim& r\frac{\left|B\right|^\alpha}{r^{n\alpha}}\int_{(B(x,4r))^c}\frac{\left|\Omega(x-z)\right|}{\left|x-z\right|^{n+1-n\alpha}}\left|f(z)-f_{B(x,r)}\right|dz\\
	\lesssim& \left|B\right|^\alpha\sum_{k=1}^{\infty}\frac{\left(2^k\right)^{n\alpha-1}}{\left|B(x,2^{k+1}r)\right|}\int_{B(x,2^{k+1}r)}\left|\Omega(x-z)\right|\left|f(z)-f_{B(x,r)}\right|dz\\
	\lesssim& \left|B\right|^\alpha\sum_{k=1}^{\infty}\Big\{\left(2^k\right)^{n\alpha-1}\Big(\frac{1}{\left|B(x,2^{k+1}r)\right|}\int_{B(x,2^{k+1}r)}\left|f(z)-f_{B(x,2^{k+1}r)}\right|^pdz\Big)^{\frac{1}{p}}\\
	&+\left(2^k\right)^{n\alpha-1}\Big(\frac{1}{\left|B(x,2^{k+1}r)\right|}\int_{B(x,2^{k+1}r)}\left|f_{B(x,2^{k+1}r)}-f_{B(x,r)}\right|^pdz\Big)^{\frac{1}{p}}\Big\}\\
	\lesssim&\left|B\right|^\alpha\sum_{k=1}^{\infty}\left(2^k\right)^{n\alpha-1}\left[\left|B(x,2^{k+1}r)\right|^\alpha+2^{kn\alpha}\left|B\right|^\alpha\right]
	\lesssim\left|B\right|^{2\alpha}.
\end{align*}
Next we consider $I_2$. Similarly as in the proof of $II$ in Theorem \ref{th1}, we take  $n\alpha<\delta=\delta_8<\frac{1-n\alpha}{2},$ since $3n\alpha-1<0$, and obtain
\begin{align*}
	I_2&\lesssim\frac{\left|B\right|^{2\alpha}}{r^{n\alpha}}\Big(\frac{\rho(x_0)}{r}\Big)^{n\alpha}\sum_{k=1}^{\infty}\int_{B(x,2^{k+1}r)\backslash B(x,2^kr)}\frac{r\left|\Omega(x-z)\right|}{\left(2^kr\right)^{n-\delta_8+1-n\alpha}\left(\rho(x)\right)^{\delta_8}}\Big(1+\frac{2^{k+1}r}{\rho(x)}\Big)^{\delta_8}dz\\
	&\lesssim\left|B\right|^{2\alpha}\Big(\frac{\rho(x_0)}{r}\Big)^{n\alpha-\delta_8}\sum_{k=1}^{\infty}(2^k)^{n\alpha+2\delta_8-1}\Big(\frac{1}{\left|B(x,2^{k+1}r)\right|}\int_{B(x,2^{k+1}r)}\left|\Omega(x-z)\right|^qdz\Big)^{\frac{1}{q}}\\
	&\lesssim\left|B\right|^{2\alpha}\Big(\frac{r}{\rho(x_0)}\Big)^{\delta_8-n\alpha}\lesssim\left|B\right|^{2\alpha}.
\end{align*}
Hence, 
\begin{align}\label{eq:th4pf11}
	\int_{8r}^{\infty}H_{f,1}(x,y,t)\left|E_f(x,t)\right|\frac{dt}{t^3}\lesssim\left|B\right|^{2\alpha}.
\end{align}
Similarly, we get
\begin{align}\label{eq:th4pf12}
	\int_{8r}^{\infty}H_{f,2}(x,y,t)\left|E_f(x,t)\right|\frac{dt}{t^3}\lesssim\left|B\right|^{2\alpha}.
\end{align}

\noindent {\bf Contributions of $H_{f,3}(x,y,t)$ and $H_{f,4}(x,y,t)$}. 
For $H_{f,3}(x,y,t)$, the fact $\mu_{j,\Omega}(1)=0$ allows us to obtain 
\begin{align}\label{eq:th5pfn2}
	H_{f,3}(x,y,t)
	\leq&\Big|\int_{E_3}\left|\Omega(x-z)\right|K{_j^\mathcal{L}}(x,z)\left[f(z)-f_{B(x,r)}\right]dz\Big|\nonumber\\
	&+\Big|\int_{E_3}\left|\Omega(x-z)\right|\left[K{_j^\mathcal{L}}(x,z)-K_j(x,z)\right]f_{B(x,r)}dz\Big|.
\end{align}
Similar argument as in Theorem \ref{th1} gives $\rho(x_0)\backsim\rho(z).$ We make use of Lemma \ref{lem:kjlguji}~(i), Lemma \ref{lem:fcampanato1}, the H\"{o}lder inequality and (\ref{def:LMCS}) to get
\begin{align*}
	&\Big|\int_{E_3}\left|\Omega(x-z)\right|K{_j^\mathcal{L}}(x,z)\left[f(z)-f_{B(x,r)}\right]dz\Big|\\&
	\lesssim\int_{8r\leq\left|x-z\right|\leq 10r}\frac{\left|\Omega(x-z)\right|}{\left|x-z\right|^{n-1}}\left|f(z)-f_{B(x,r)}\right|dz\\&
	\lesssim\frac{1}{r^{n-1}}\int_{\left|x-z\right|\leq10r}\left|\Omega(x-z)\right|\left\{\left|f(z)-f_{B(x,10r)}\right|+\left|f_{B(x,10r)}-f_{B(x,r)}\right|\right\}dz\\&
	\lesssim r\Big(\frac{1}{\left|B(x,10r)\right|}\int_{\left|x-z\right|\leq 10r}\left|f(z)-f_{B(x,10r)}\right|^{p}dz\Big)^{\frac{1}{p}}+r\left|B\right|^\alpha\lesssim r\left|B\right|^\alpha.
\end{align*}
Applying Lemma \ref{lem:kjlguji} (iii), Lemma \ref{lem:fcampanato1} and the H\"{o}lder inequality, and taking $\delta=\delta_7>n\alpha$, we obtain 
\begin{align*}
	&\Big|\int_{E_3}\left|\Omega(x-z)\right|\left[K{_j^\mathcal{L}}(x,z)-K_j(x,z)\right]f_{B(x,r)}dz\Big|\\&
	\lesssim\int_{8r\leq\left|x-z\right|\leq 10r}\frac{\left|\Omega(x-z)\right|}{\left|x-z\right|^{n-\delta_7-1}\left[\rho(z)\right]^{\delta_7}}\left|f_{B(x,r)}\right|dz\\&
	\lesssim r\left(\frac{r}{\rho(x_0)}\right)^{\delta_7-n\alpha}\left|B\right|^\alpha\Big(\frac{1}{\left|B(x,10r)\right|}\int_{\left|x-z\right|\leq 10r}\left|\Omega(x-z)\right|^qdz\Big)^{\frac{1}{q}}
	\lesssim r\left|B\right|^\alpha.
\end{align*}
Thus, when $n\alpha-1<0$, (\ref{eq:th4pf5}) and (\ref{eq:th5pfn2}) imply that
\begin{align}\label{eq:th4pf9}
	\int_{8r}^{\infty}H_{f,3}(x,y,t)\left|E_f(x,t)\right|\frac{dt}{t^3}\lesssim\int_{8r}^{\infty}r\left|B\right|^\alpha t\left(\frac{t}{r}\right)^{n\alpha}\left|B\right|^\alpha\frac{dt}{t^3}\lesssim\left|B\right|^{2\alpha}.
\end{align}
Similarly, we have 
\begin{align}\label{eq:th4pf10}
	\int_{8r}^{\infty}H_{f,4}(x,y,t)\left|E_f(x,t)\right|\frac{dt}{t^3}\lesssim\left|B\right|^{2\alpha}.
\end{align}

\noindent {\bf Contributions of $H_{f,5}(x,y,t)$. }  Now it remains to consider $H_{f,5}(x,y,t)$. Note that
\begin{align*}
	&\left[\left|\Omega(x-z)\right|K{_j^\mathcal{L}}(x,z)-\left|\Omega(y-z)\right|K{_j^\mathcal{L}}(y,z)\right]f(z)\\&
	=\left[\left|\Omega(x-z)\right|-\left|\Omega(y-z)\right|\right]K{_j^\mathcal{L}}(y,z)f(z)\\&\quad+\left|\Omega(x-z)\right|\left[K{_j^\mathcal{L}}(x,z)-K{_j^\mathcal{L}}(y,z)\right](f(z)-f_{B(x,r)})\\
	&\quad+\left|\Omega(x-z)\right|\left\{[K{_j^\mathcal{L}}(x,z)-K_{j}(x,z)]-[K{_j^\mathcal{L}}(y,z)-K_{j}(y,z) ]\right\}f_{B(x,r)}\\
	&\quad+\left|\Omega(x-z)\right|\left[K_j(x,z)-K_{j}(y,z)\right]f_{B(x,r)}.
\end{align*}
Let $y-z=h\theta\in(8r,t)\times \mathbb{S}^{n-1}$. Since $x,y\in B,~8r\leq\left|x-z\right|\leq t,~ 8r\leq\left|y-z\right|\leq t$, for any orthogonal transformation T, there exists $\tau\in\mathbb{R}$ such that $x-z=\tau hT\theta\in(8r,t)\times \mathbb{S}^{n-1}$. Note that $\Omega$ is homogeneous of degree zero, it follows 
\begin{align}\label{eq:th4pf14}
	&\int_{E_5}\left|\Omega(x-z)\right|K_j(y,z)dz=\int_{8r}^{t}\int_{\mathbb{S}^{n-1}}\left|\Omega(T\theta)\right|\theta_jd\theta dh=0.
\end{align}
Combining this with the fact $\mu_{j,\Omega}(1)=0$ and Lemma \ref{lem:kernelguji}, we take $s=m$, where the range of $m$ will be given later, we get
\begin{align*}
	&\int_{8r}^{\infty}H_{f,5}(x,y,t)\left|E_f(x,t)\right|\frac{dt}{t^3}\nonumber\\&
	\lesssim\frac{\left|B\right|^\alpha}{r^{n\alpha}}\int_{8r}^{\infty}\left|\int_{E_5}\left[\left|\Omega(x-z)\right|-\left|\Omega(y-z)\right|\right]K{_j^\mathcal{L}}(y,z)f(z)dz\right|\frac{dt}{t^{2-n\alpha}}\nonumber\\
	&+\frac{\left|B\right|^\alpha}{r^{n\alpha}}\int_{8r}^{\infty}\left|\int_{E_5}\left|\Omega(x-z)\right|\left[K{_j^\mathcal{L}}(x,z)-K{_j^\mathcal{L}}(y,z)\right]\left[f(z)-f_{B(x,r)}\right]dz\right|\frac{dt}{t^{2-n\alpha}}\nonumber\\
	&+\left[\rho(x)\right]^m\frac{\left|B\right|^\alpha}{r^{n\alpha}}\int_{8r}^{\infty}\bigg|\int_{E_5}\left|\Omega(x-z)\right| K_j(x,y,z) f_{B(x,r)}dz\bigg|\frac{dt}{t^{2-n\alpha+m}}\nonumber\\
	=:&J_1+J_2+J_3.
\end{align*}
By $(\ref{eq:th4pf14})$ and $\mu_{j,\Omega}(1)=0$ again, we may bounded $J_1$ in the way
\begin{align*}
J_1\leq&\frac{\left|B\right|^\alpha}{r^{n\alpha}}\int_{8r}^{\infty}\Big|\int_{E_5}\left[\left|\Omega(x-z)\right|-\left|\Omega(y-z)\right|\right]K{_j^\mathcal{L}}(y,z)(f(z)-f_{B(y,r)})dz\Big|\frac{dt}{t^{2-n\alpha}}\\
	&+\frac{\left|B\right|^\alpha}{r^{n\alpha}}\int_{8r}^{\infty}\int_{E_5}\left[\left|\Omega(x-z)\right|-\left|\Omega(y-z)\right|\right]\left[K{_j^\mathcal{L}}(y,z)-K{_j}(y,z)\right]f_{B(y,r)}dz\Big|\frac{dt}{t^{2-n\alpha}}\\
	=:&J_{11}+J_{12}.
\end{align*}
Since $x,y\in B(x_0,r),~8r\leq\left|x-z\right|\leq t, ~8r\leq\left|y-z\right|\leq t$, we obtain $\left|x-z\right|\backsim\left|y-z\right|.$ Next, we will use the $L^q$-Dini typr condition (\ref{def:lqDini2}) to estimate $J_{11}$ and $J_{12}$ respectively. Let us first consider $J_{11}$. Using the Minkowski inequality, we have 
\begin{align*}
J_{11}\leq&\frac{\left|B\right|^\alpha}{r^{n\alpha}}\int_{(B(y,8r))^c}\big|\left|\Omega(x-z)\right|-\left|\Omega(y-z)\right|\big|\left|K{_j^\mathcal{L}}(y,z)\right|\left|f(z)-f_{B(y,r)}\right|\int_{\left|x-z\right|\leq t\atop \left|y-z\right|\leq t}\frac{dt}{t^{2-n\alpha}}dz\\
	\lesssim&\frac{\left|B\right|^\alpha}{r^{n\alpha}}\int_{(B(y,8r))^c}\left|\frac{\left|\Omega(x-z)\right|}{\left|x-z\right|^{n-n\alpha}}-\frac{\left|\Omega(y-z)\right|}{\left|x-z\right|^{n-n\alpha}}\right|\left|f(z)-f_{B(y,r)}\right|dz.
	\end{align*}
Similarly as we what we have done with (\ref{eq:th2pf11}) gives
\begin{align*}
		&\left|\frac{\left|\Omega(x-z)\right|}{\left|x-z\right|^{n-n\alpha}}-\frac{\left|\Omega(y-z)\right|}{\left|x-z\right|^{n-n\alpha}}\right|\lesssim\left|\frac{\left|\Omega(x-z)\right|}{\left|x-z\right|^{n-n\alpha}}-\frac{\left|\Omega(y-z)\right|}{\left|y-z\right|^{n-n\alpha}}\right|+\frac{r\left|\Omega(y-z)\right|}{\left|y-z\right|^{n+1-n\alpha}}
\end{align*}
Therefore
\begin{align*}
	J_{11}\lesssim&r\frac{\left|B\right|^\alpha}{r^{n\alpha}}\int_{(B(y,4r))^c}\frac{\left|\Omega(y-z)\right|}{\left|y-z\right|^{n+1-n\alpha}}\left|f(z)-f_{B(y,r)}\right|dz\\&
	+\frac{\left|B\right|^\alpha}{r^{n\alpha}}\int_{(B(y,8r))^c}\left|\frac{\left|\Omega(x-z)\right|}{\left|x-z\right|^{n-n\alpha}}-\frac{\left|\Omega(y-z)\right|}{\left|y-z\right|^{n-n\alpha}}\right|\left|f(z)-f_{B(y,r)}\right|dz\\
	&=:J_{11}^1+J_{11}^2.
\end{align*}
Exactly the same argument for $I$ applies to $J_{11}^1$, replacing $x$ with $y$ yields $J_{11}^1\lesssim\left|B\right|^{2\alpha}$. For  $	J_{11}^2$, it follows from lemma \ref{lem:kjlguji}~(i) and lemma \ref{lm:omegalqdini} that
\begin{align*}
	J_{11}^2&
	\leq\frac{\left|B\right|^\alpha}{r^{n\alpha}}\sum_{k=3}^{\infty}\Big(\int_{2^kr\leq\left|y-z\right|\leq2^{k+1}r}\left|\frac{\left|\Omega(x-z)\right|}{\left|x-z\right|^{n-n\alpha}}-\frac{\left|\Omega(y-z)\right|}{\left|y-z\right|^{n-n\alpha}}\right|^qdz\Big)^{\frac{1}{q}}\\
	&\quad\times\Big(\int_{2^kr\leq\left|y-z\right|\leq2^{k+1}r}\left|f(z)-f_{B(y,r)}\right|^{q'}dz\Big)^{\frac{1}{q'}}\\
	\lesssim&\frac{\left|B\right|^\alpha}{r^{n\alpha}}\sum_{k=3}^{\infty}\left(2^kr\right)^{n\alpha}\frac{1}{2^{k}}\Big(\frac{1}{\left(2^{k+1}r\right)^n}\int_{2^kr\leq\left|y-z\right|\leq2^{k+1}r}\left|f(z)-f_{B(y,r)}\right|^{q'}dz\Big)^{\frac{1}{q'}}\\
	&+\frac{\left|B\right|^\alpha}{r^{n\alpha}}\sum_{k=3}^{\infty}\left(2^kr\right)^{n\alpha}\Big(\int_{\frac{\left|x-y\right|}{2^{k+1}r}}^{\frac{\left|x-y\right|}{2^kr}}\frac{\omega_q(\sigma)}{\sigma}d\sigma\Big)\Big(\frac{1}{\left(2^{k+1}r\right)^n}\int_{2^kr\leq\left|y-z\right|\leq2^{k+1}r}\left|f(z)-f_{B(y,r)}\right|^{q'}dz\Big)^{\frac{1}{q'}}
	\end{align*}
Note that when $0<\epsilon\leq1$, we have $$\int_{\frac{\left|x-y\right|}{2^{k+1}r}}^{\frac{\left|x-y\right|}{2^{k}r}} \frac{\omega_q(\sigma)}{\sigma} d\sigma \leq \frac{1}{2^{k\epsilon}} \int_0^1 \frac{\omega_q(\sigma)}{\sigma^{1+\epsilon}} d \sigma.$$ Thus, when $2n\alpha-\epsilon<0$,
\begin{align*}  
J_{11}^2&\lesssim\left|B\right|^\alpha\sum_{k=3}^{\infty}\Big\{\frac{1}{\left(2^k\right)^{1-n\alpha}}+\frac{\left(2^k\right)^{n\alpha}}{\left(2^k\right)^{\epsilon}}\int_{0}^{1}\frac{\omega_q(\sigma)}{\sigma^{1+\epsilon}}d\sigma\Big\}\\&\quad\times\Big(\frac{1}{\left(2^{k+1}r\right)^n}\int_{2^kr\leq\left|y-z\right|\leq2^{k+1}r}\left|f(z)-f_{B(y,r)}\right|^{p}dz\Big)^{\frac{1}{p}}\\
	\lesssim&\left|B\right|^{2\alpha}\Big\{\sum_{k=3}^{\infty}\frac{1}{\left(2^k\right)^{1-2n\alpha}}+\sum_{k=3}^{\infty}\frac{1}{\left(2^k\right)^{\epsilon-2n\alpha}}\Big\}
	\lesssim\left|B\right|^{2\alpha}.
\end{align*}
For the term $J_2$. Similar as the estimate to $J_1$, we use the Minkowski inequality, Lemma \ref{lem:fuzhuhanshu}, Lemma \ref{lem:kjlguji} (iii) and dominate $	J_2$ by
\begin{align*}
&\frac{\left|B\right|^\alpha}{r^{n\alpha}}\int_{(B(y,8r))^c}\big|\left|\Omega(x-z)\right|-\left|\Omega(y-z)\right|\big|\left|K{_j^\mathcal{L}}(y,z)-K{_j}(y,z)\right|\left|f_{B(y,r)}\right|\int_{\left|x-z\right|\leq t\atop \left|y-z\right|\leq t}\frac{dt}{t^{2-n\alpha}}dz\\&
	\lesssim\frac{\left|B\right|^{2\alpha}}{r^{n\alpha}}\Big(\frac{\rho(x)}{r}\Big)^{n\alpha}\Big(\frac{1}{\rho(x)}\Big)^{\delta}\int_{(B(y,8r))^c}\bigg|\frac{\left|\Omega(x-z)\right|}{\left|x-z\right|^{n-\delta-n\alpha}}-\frac{\left|\Omega(y-z)\right|}{\left|x-z\right|^{n-\delta-n\alpha}}\bigg|\Big(1+\frac{\left|y-z\right|}{\rho(y)}\Big)^{l_0\delta}dz\\&
	\lesssim\frac{\left|B\right|^{2\alpha}}{r^{n\alpha}}\Big(\frac{\rho(x)}{r}\Big)^{n\alpha}\Big(\frac{1}{\rho(x)}\Big)^{\delta}\int_{(B(y,8r))^c}\bigg|\frac{\left|\Omega(x-z)\right|}{\left|x-z\right|^{n-\delta-n\alpha}}-\frac{\left|\Omega(y-z)\right|}{\left|y-z\right|^{n-\delta-n\alpha}}\bigg|\Big(1+\frac{\left|y-z\right|}{\rho(y)}\Big)^{\delta}dz\\
	&+\frac{\left|B\right|^{2\alpha}}{r^{n\alpha}}\Big(\frac{\rho(x)}{r}\Big)^{n\alpha}\Big(\frac{1}{\rho(x)}\Big)^{\delta}\int_{(B(y,4r))^c}\frac{r\left|\Omega(y-z)\right|}{\left|y-z\right|^{n-\delta+1-n\alpha}}\Big(1+\frac{\left|y-z\right|}{\rho(y)}\Big)^{\delta}dz=:J_{21}+J_{22}.
\end{align*}
It can be deduced in the same way as $I_2$ to get $J_{22}\leq\left|B\right|^{2\alpha}$. As for $J_{21}$, Lemma \ref{lm:omegalqdini} yields
\begin{align*}
J_{21}\lesssim&\frac{\left|B\right|^{2\alpha}}{r^{n\alpha}}\Big(\frac{\rho(x)}{r}\Big)^{n\alpha}\Big(\frac{1}{\rho(x)}\Big)^{\delta}\sum_{k=3}^{\infty}2^{k\delta}\\
	&\times\bigg[\int_{2^kr\leq\left|y-z\right|\leq2^{k+1}r}\bigg|\frac{\left|\Omega(x-z)\right|}{\left|x-z\right|^{n-\delta-n\alpha}}-\frac{\left|\Omega(y-z)\right|}{\left|y-z\right|^{n-\delta-n\alpha}}\bigg|^qdz\bigg]^{\frac{1}{q}}\Big(\int_{2^kr\leq\left|y-z\right|\leq2^{k+1}r}dz\Big)^{\frac{1}{q'}}\\
	\leq&\frac{\left|B\right|^{2\alpha}}{r^{n\alpha}}\Big(\frac{\rho(x)}{r}\Big)^{n\alpha}\Big(\frac{1}{\rho(x)}\Big)^{\delta}\sum_{k=3}^{\infty}2^{k\delta}\left(2^kr\right)^{\frac{n}{q}-(n-\delta-n\alpha)}\Big\{\frac{\left|x-y\right|}{2^{k}r}+\int_{\frac{\left|x-y\right|}{2^{k+1}r}}^{\frac{\left|x-y\right|}{2^kr}}\frac{\omega_q(\sigma)}{\sigma}d\sigma\Big\}\\
	&\times\Big(\int_{2^kr\leq\left|y-z\right|\leq2^{k+1}r}dz\Big)^{\frac{1}{q'}}
\end{align*}
In this case, we take $\delta=\delta_9=2n\alpha>n\alpha$ when $\alpha<\epsilon/{3n}$ and obtain 
\begin{align*}
	J_{21}\lesssim&\frac{\left|B\right|^{2\alpha}}{r^{n\alpha}}\Big(\frac{\rho(x)}{r}\Big)^{n\alpha}\Big(\frac{1}{\rho(x)}\Big)^{\delta_9}\sum_{k=3}^{\infty}2^{k\delta_9}\left(2^kr\right)^{\delta_9+n\alpha}\Big\{\frac{1}{2^{k}}+\frac{1}{\left(2^k\right)^{\epsilon}}\int_{0}^{1}\frac{\omega_q(\sigma)}{\sigma^{1+\epsilon}}d\sigma\Big\}\\
	&\times\Big(\frac{1}{\left(2^{k+1}r\right)^n}\int_{2^kr\leq\left|y-z\right|\leq2^{k+1}r}dz\Big)^{\frac{1}{q'}}\\
	\lesssim&\left|B\right|^{2\alpha}\Big(\frac{\rho(x)}{r}\Big)^{n\alpha}\Big(\frac{r}{\rho(x)}\Big)^{\delta_9}\Big\{\sum_{k=3}^{\infty}\frac{1}{\left(2^k\right)^{1-2\delta_9-n\alpha}}+\sum_{k=3}^{\infty}\frac{1}{\left(2^k\right)^{\epsilon-2\delta_9-n\alpha}}\Big\}\\
	\lesssim&\left|B\right|^{2\alpha}\Big(\frac{r}{\rho(x)}\Big)^{\delta_9-n\alpha}\lesssim\left|B\right|^{2\alpha}.
\end{align*}
Let us consider $J_2$. By the Minkowski  inequality, Lemma \ref{lem:kjlguji}~(ii), the H\"{o}lder inequality and Lemma \ref{lem:fcampanato1}, we take $\beta>2n\alpha$ and obtain 
\begin{align*}
J_2
	\lesssim& r^\beta\frac{\left|B\right|^\alpha}{r^{n\alpha}}\int_{(B(x,4r))^c}\frac{\left|\Omega(x-z)\right|\left|f(z)-f_{B(x,r)}\right|}{\left|x-z\right|^{n-1+\beta}}\Big(\int_{\left|x-z\right|}^{\infty}\frac{dt}{t^{2-n\alpha}}\Big)dz\\
	\lesssim& \left|B\right|^\alpha\sum_{k=1}^{\infty}\left(2^k\right)^{n\alpha-\beta}\Big(\frac{1}{\left|B(x,2^{k+1}r)\right|}\int_{B(x,2^{k+1}r)}\left|f(z)-f_{B(x,r)}\right|^{p}dz\Big)^{\frac{1}{p}}\\
	\lesssim& \left|B\right|^\alpha\sum_{k=1}^{\infty}\bigg\{\left(2^k\right)^{n\alpha-\beta}\Big(\frac{1}{\left|B(x,2^{k+1}r)\right|}\int_{B(x,2^{k+1}r)}\left|f(z)-f_{B(x,2^{k+1}r)}\right|^{p}dz\Big)^{\frac{1}{p}}\\
	&+\left(2^k\right)^{n\alpha-\beta}\Big(\frac{1}{\left|B(x,2^{k+1}r)\right|}\int_{B(x,2^{k+1}r)}\left|f_{B(x,2^{k+1}r)}-f_{B(x,r)}\right|^{p}dz\Big)^{\frac{1}{p}}\bigg\}\\
	\lesssim&\left|B\right|^\alpha\sum_{k=1}^{\infty}\left(2^k\right)^{n\alpha-\beta}\left[\left|B(x,2^{k+1}r)\right|^\alpha+2^{kn\alpha}\left|B\right|^\alpha\right]\lesssim\left|B\right|^{2\alpha}.
\end{align*}
Finally, it remains to consider $J_3$. Applying (\ref{eq:keyxfj}), $J_3$ can be controlled by two terms.
\begin{align*}
	{J_3}\lesssim&\frac{\left[\rho(x)\right]^m\left|B\right|^\alpha}{r^{n\alpha}}\int_{8r}^{\infty}\Big|\int_{E_5}\frac{\left|\Omega(x-z)\right||x-y|^\gamma}{|x-z|^{n+\gamma-1}}\left(\frac{|x-z|}{\rho(x)}\right)^{\eta}\left[ f_{B(x,r)}\right]dz\Big|\frac{dt}{t^{2-n\alpha+m}}\\
	&+\frac{\left[\rho(x)\right]^m\left|B\right|^\alpha}{r^{n\alpha}}\int_{8r}^{\infty}\int_{E_5}\frac{r\left|\Omega(x-z)\right|}{|x-z|^{n}}\left(\frac{|x-z|}{\rho(z)}\right)^\delta\left[ f_{B(x,r)}\right]dz\Big|\frac{dt}{t^{2-n\alpha+m}}\\
	=:&J_{31}+J_{32}.
\end{align*}
In Lemma \ref{lem:kernelguji2}, we take $s=m$, where $m$ satisfies $0<m=\eta-\frac{\gamma}{2}<1$. When $2n\alpha-\gamma<0$, we have $n\alpha-\eta+m<0$ and $n\alpha+\eta-\gamma-m<0$. By the Minkowski inequality, the H\"{o}lder inequality and Remark \ref{lem:kjlcha2}, we obtain 
\begin{align*}
J_{31}\lesssim&\left[\rho(x)\right]^m\frac{\left|B\right|^{2\alpha}}{r^{n\alpha}}\Big(\frac{\rho(x_0)}{r}\Big)^{n\alpha}\int_{(B(x,4r))^c}\frac{\left|\Omega(x-z)\right||x-y|^\gamma}{|x-z|^{n+\gamma-1}}\Big(\frac{|x-z|}{\rho(x)}\Big)^{\eta}\Big(\int_{\left|x-z\right|}^{\infty}\frac{dt}{t^{2-n\alpha+m}}\Big)dz\\
	\lesssim&\left[\rho(x)\right]^m\frac{\left|B\right|^{2\alpha}}{r^{n\alpha}}\Big(\frac{\rho(x_0)}{r}\Big)^{n\alpha}\sum_{k=1}^{\infty}\int_{B(x,2^{k+1}r)\backslash B(x,2^kr)}\frac{\left|\Omega(x-z)\right|r^{\gamma}}{\left(2^kr\right)^{n+\gamma-\eta+m-n\alpha}\left(\rho(x)\right)^\eta}dz\\
	\lesssim&\left[\frac{\rho(x)}{r}\right]^{m+n\alpha-\eta}\left|B\right|^{2\alpha}\sum_{k=1}^{\infty}(2^k)^{n\alpha+\eta-\gamma-m}\Big(\frac{1}{\left|B(x,2^{k+1}r)\right|}\int_{B(x,2^{k+1}r)}\left|\Omega(x-z)\right|^qdz\Big)^{\frac{1}{q}}\\
	\lesssim&\left|B\right|^{2\alpha}\Big(\frac{r}{\rho(x_0)}\Big)^{\eta-n\alpha-m}
	\lesssim\left|B\right|^{2\alpha}.
\end{align*}
For $J_{32}$, we pick $m=0$.The Minkowski integral inequality then leads to
\begin{align*}
J_{32}&\lesssim\frac{\left|B\right|^{2\alpha}}{r^{n\alpha}}\Big(\frac{\rho(x_0)}{r}\Big)^{n\alpha}\int_{(B(x,4r))^c}\frac{r\left|\Omega(x-z)\right|}{\left|x-z\right|^{n-\delta+1-n\alpha}\left(\rho(z)\right)^{\delta}}dz
\end{align*}
Using the same method as $I_2$ to prove $J_{32}$ with $m=0$, it follows $J_{32}\lesssim\left|B\right|^{2\alpha}$ and hence $J_3\lesssim\left|B\right|^{2\alpha}$. Therefore 
\begin{align}\label{eq:th4pf13}
	\int_{8r}^{\infty}H_{f,5}(x,y,t)\left|E_f(x,t)\right|\frac{dt}{t^3}\lesssim\left|B\right|^{2\alpha}.
\end{align}
By (\ref{eq:th4pf11}), (\ref{eq:th4pf12}) , (\ref{eq:th4pf9}), (\ref{eq:th4pf10}) and (\ref{eq:th4pf13}), we have (\ref{eq:th4pf8}).
Combining  (\ref{eq:th4pf8}) with (\ref{eq:th4pf6}) and (\ref{eq:th4pf7}), we obtain  (\ref{eq:th3pf9}). The proof of Theorem \ref{th3} is finished.
\qed

 \end{document}